\documentclass[12pt]{article}
\usepackage{amssymb,amsmath, amsthm, mathrsfs}
\numberwithin{equation}{section}
\newtheorem{prop}{Proposition}[section]
\newtheorem{defi}{Definition}[section]
\newtheorem{lemma}{Lemma}[section]
\newtheorem{theorem}{Theorem}[section]
\newtheorem{conj}{Conjecture}[section]
\newtheorem*{nonnumber}{Theorem} 
\def\ol{\overline}

\allowdisplaybreaks

\begin{document}

\title{On compact splitting complex submanifolds \\ of quotients of bounded symmetric domains}

\author{Ngaiming Mok\footnote{Department of Mathematics, The University of Hong Kong, Hong Kong, People's Republic of China. \textbf{Email:}~nmok@hku.hk}, Sui-Chung Ng\footnote{Department of Mathematics, Shanghai Key Laboratory of PMMP, East China Normal University, Shanghai, People's Republic of China. \textbf{Email:}~scng@math.ecnu.edu.cn}}

\date{}

\maketitle

\begin{abstract}
  In the current article our primary objects of study are compact complex submanifolds $S$ of quotient manifolds $X = \Omega/\Gamma$ of irreducible bounded symmetric domains by torsion free discrete lattices of  automorphisms, and we are interested in the characterization of the totally geodesic submanifolds among compact  splitting complex submanifolds 
$S\subset X$, i.e., under the assumption that the tangent sequence over
$S$ splits holomorphically.  We prove results of two type.  The first type of results concerns $S \subset X$ which are characteristic complex submanifolds, i.e., embedding $\Omega$ as an open subset of its compact dual
manifold $M$ by means of the Borel embedding, the non-zero $(1,0)$-vectors tangent to $S$ lift under a local  inverse of the universal covering map $\pi: \Omega \to X$ to minimal rational tangents of $M$.  We prove that  a compact characteristic complex submanifold $S \subset X$ is necessarily totally geodesic whenever $S$ is a  splitting complex submanifold.  Our proof generalizes the case of the characterization of totally geodesic complex submanifolds of quotients of the complex unit ball $\mathbb B^n$ obtained in~\cite{s16}.  The proof given  here is however new
and it is based on a monotonic property of curvatures of
 Hermitian holomorphic vector subbundles of Hermitian holomorphic vector bundles and on exploiting the splitting  of the tangent sequence to identify the holomorphic tangent bundle $T_S$ as a quotient bundle rather than as a  subbundle of the restriction of the holomorphic tangent bundle
$T_X$ to $S$. The second type of results concerns characterization of total geodesic submanifolds among compact splitting complex submanifolds $S
 \subset X$ deduced from  the results of~\cite{s1} and~\cite{s25} which imply the existence of K\"ahler-Einstein metrics on  $S \subset X$.  We prove that compact splitting complex submanifolds $S \subset X$ of sufficiently large dimension (depending on $\Omega$) are necessarily totally geodesic.  The proof relies on the  Hermitian-Einstein property of holomorphic vector bundles associated to $T_S$, which implies that  endomorphisms of such bundles are parallel, and the construction of endomorphisms
of these vector bundles by means of the splitting of the  tangent sequence on $S$.  We conclude with conjectures on the sharp lower bound on $\dim(S)$ guaranteeing  total geodesy of $S \subset X$ for the case of type-I domains of rank 2 and the case of type-IV domains, and  examine a case which is critical for both conjectures, viz. on compact complex surfaces of quotients of the  4-dimensional Lie ball, equivalently the 4-dimensional type-I domain dual to the Grassmannian of 2-planes  in $\mathbb C^4$.
\end{abstract}

\section{Introduction}

Our primary objects of study in this article are compact complex submanifolds $S$ of quotients $X = \Omega/\Gamma$ of irreducible bounded symmetric domains, and we are interested in the characterization of $S \subset X$ which are totally geodesic. When $S \subset X$ is totally geodesic, the tangent sequence $0 \to T(S) \to T(X)|_S \to N_{S|X} \to 0$ splits holomorphically.  In general, for a complex submanifold $S \subset X$, we say that $S$ is a splitting complex submanifold if the tangent sequence splits holomorphically over $S$, and in the current article we consider the question of characterizing among compact splitting complex submanifolds $S \subset X = \Omega/\Gamma$ those which are totally geodesic. The first result about compact splitting submanifolds concerns the projective space $\mathbb P^n$ (see~\cite{s22}), which is dual to the complex unit ball $\mathbb B^n$, and the result is that any compact splitting submanifold of $\mathbb P^n$ is linear, hence totally geodesic with respect to (any choice of) the Fubini-Study metric. In~\cite{s16}  the first author studied the problem from a differential-geometric perspective, proving the same simultaneously for $\mathbb P^n$, $\mathbb B^n$ and the $n$-dimensional compact complex torus $T = \mathbb C^n/L$ by exploiting the canonical K\"ahler-Einstein metric and the projective connection.

In the current article we are only concerned with the case of Hermitian locally symmetric spaces of the noncompact type, and, generalizing the result in~\cite{s16} for the complex unit ball $\mathbb B^n$ to an irreducible bounded symmetric domain $\Omega$. We prove first of all that a compact splitting complex submanifold $S \subset X := \Omega/\Gamma$ is totally geodesic whenever $S \subset X$ is a characteristic complex submanifold, which means that $S$ is tangent at each point to a local totally geodesic complex submanifold of a special kind, as follows. Embed $\Omega$ as an open subset of its dual Hermitian symmetric manifold $Z$ of the compact type, and identify $Z$ as a projective submanifold by means of the minimal embedding. Then, $Z \subset \mathbb P(\Gamma(Z,{\cal O}(1))^*) =: \mathbb P^N$ is uniruled by projective lines, and, defining the subset $\mathscr C (Z) \subset \mathbb PT(Z)$ to consist at each point $z \in Z$ of projectivizations of tangents of projective lines (minimal rational curves), we have on $Z$ a VMRT (variety of minimal rational tangents) structure $\pi_Z: {\mathscr C} (Z) \to Z$, which restricts to $\Omega$ and descends by Aut$(\Omega)$-invariance to $X = \Omega/\Gamma$, defining thus $\pi_X: {\mathscr C} (X) \to X$, and we say that $S \subset X$ is a characteristic complex submanifold to mean that $\mathbb PT(S) \subset {\mathscr C} (X)$. Thus, $S$ is tangent at each point to a local complex submanifold which lifts to an open subset of a projective linear subspace $\Lambda$ of $Z \subset \mathbb P^N$. In a certain sense, we have a characterization of ``{\it linear\,}'' geodesic submanifolds $S \subset X$ since $\Lambda \subset Z \subset \mathbb P^N$ is linear, and since moreover the intersection $\Lambda \cap \Omega$ is a connected open subset of the affine linear subspace $\Lambda \cap \mathbb C^n$, when we identify $\Omega$ as an open subset of $\mathbb C^n$ by means of the Harish-Chandra embedding. Normalizing the choice of canonical K\"ahler-Einstein metric $g_\Omega$ on $\Omega$ to be such that minimal disks are of Gaussian curvature $-2$, then $\Lambda \cap \Omega \subset \Omega$ is a totally geodesic Hermitian symmetric submanifold of constant holomorphic sectional curvature $-2$. (In particular $(\Lambda \cap \Omega, g_\Omega|_{\Lambda \cap \Omega})$ is holomorphically isometric to the complex hyperbolic space form $(\mathbb B^n,g_{\mathbb B^n})$). Our proof also exploits the canonical K\"ahler-Einstein metric of $\Omega$, but is otherwise elementary relying only on a monotonic property of the curvature of holomorphic vector subbundles of Hermitian holomorphic vector bundles, which results from the Gauss equation.

While the rank 1 case pertains to the projective structure, the higher rank case of $\Omega \subset Z$ underlies geometric structures modeled on reductive Lie groups.  These holomorphic G-structures are captured again by the canonical K\"ahler-Einstein metric, and we make use of them to characterize not necessarily ``{\it linear\,}'' compact totally geodesic complex submanifolds $S \subset X$ under dimension restrictions, noting that there exist examples of non-totally geodesic compact splitting submanifolds $S \subset X$ given by graphs of surjective holomorphic maps which are not covering maps between compact ball quotients. When $\Omega$ is of rank $\ge 2$, in our proof on the one hand we make use of the underlying G-structure on $\Omega$ which leads to non-trivial direct sum decompositions of associated vector bundles of the tangent bundle $T(X)$, especially $S^2T(X)$ and End$_0(T(X))$, on the other hand we make use of the holomorphic splitting of the tangent sequence of the inclusion 
$S \subset X$ to obtain endomorphisms of $S^2(T(S))$ and End$_0(T(S))$. Since the canonical line bundle of a compact complex submanifold $S \subset X$ is ample, there exists on $S$ a K\"ahler-Einstein metric, and hence there exists a Hermitian-Einstein metric $h$ on any associated tensor bundle $V$ of $T(S)$, from which it follows that the kernel of any endomorphism of $V$ over $S$ must be a parallel subbundle.  Making use of this basic principle we prove in \S4 the total geodesy of $S$ under certain dimension restrictions on $S$.

In \S5 we consider two special classes of bounded symmetric domains, viz., the case of rank-2 type-I bounded symmetric domains $D^I_{2,p}$ and the case of type-IV domains $D^{IV}_n$ (Lie balls), $n \ge 3$, showing that in the case of rank-2 type-I domains a compact splitting complex submanifold $S \subset D^I_{2,p}/\Gamma$ must be totally geodesic whenever $\dim(S) > p$ and that in the case of type-IV domains the same holds true for a compact splitting complex submanifold $S \subset D^{IV}_n/\Gamma$ whenever $\dim(S) > \frac {n}{\sqrt{2}}$.  For the former case we conjecture that $\dim(S) \ge p$ suffices, while for the latter case we conjecture that $\dim(S) \ge 2$ suffices. For the rank-2 type-I case there exist $(p-1)$-dimensional splitting and non-totally geodesic compact complex submanifolds $S \subset D^I_{2,p}/\Gamma$ at least in the case where $p = 2, 3, 4$ while in the type-IV case the dual problem for splitting complex submanifolds $S \subset Q^n$ of the hyperquadrics $Q^n$ has been solved by~\cite{s10}, where $S$ was shown to be a linear subspace or a smooth hyperquadric obtained as a linear section of $Q^n$.  In \S6, the last section, we examine the borderline case of $D^I_{2,2} \cong D^{IV}_4$, and we prove that in this case any compact splitting complex submanifold of dimension $\ge 2$ must necessarily either be totally geodesic, or, writing $T(X)|_S = T_S \oplus {\cal N}$ for the holomorphic splitting of the tangent sequence, the complementary bundle ${\cal N}$ must consist of characteristic $2$-planes, an intermediate result leaving open the delicate problem of characterizing splitting compact complex surfaces with characteristic complementary bundles ${\cal N}$ consisting of characteristic 2-planes, i.e., 
$\mathbb P({\cal N}) \subset {\mathscr C} (X)|_S$.

\section{Background materials and statements of results}

In this section we state the main results and collect background materials necessary for the understanding of the statements of these results.  The main results and the necessary background materials break down into those concerning characteristic compact complex submanifolds and those related to the use of K\"ahler-Einstein metrics for compact complex submanifolds of quotients of bounded symmetric domains.

\subsection
{Characterization of compact totally geodesic submanifolds among characteristic complex submanifolds on quotients of irreducible bounded symmetric domains}  
To start with we define ``splitting complex manifolds'' which are the primary objects of study of the current article.

\begin{defi}[{\bf splitting complex submanifold}]
Let $X$ be a complex manifold and $S \subset X$ be a complex submanifold. We say that $S$ is a splitting complex submanifold in $X$ if and only if the tangent sequence $0 \to T(S) \to T(X)|_S \to N_{S|X} \to 0$ splits holomorphically over $S$.
\end{defi}

Concerning compact splitting complex submanifolds the first result was a result of~\cite{s22} characterizing projective linear subspaces of the projective space as precisely the compact splitting complex submanifolds (see \,also~\cite{s18}). Compact splitting complex submanifolds of compact complex tori are also characterized in~\cite{s10}.  In~\cite{s16} the first author of the current article examined compact splitting complex submanifolds in complex hyperbolic space forms (i.e., complex manifolds uniformized by the complex unit ball $\mathbb B^n$), complex Euclidean space forms and the projective space and proved using differential-geometric method the following characterization theorem.

\begin{nonnumber}
{\rm (See [16, Theorem 1]).}
Let $(X,g)$ be a complex hyperbolic space form, a complex Euclidean space form or the complex projective space endowed with the Fubini-Study metric.
Let $S \subset X$ be a compact complex submanifold. Then $S \subset X$ is a splitting complex submanifold if and only if $(S,g|_S) \ \hookrightarrow (X,g)$ is totally geodesic.
\end{nonnumber}

In [16, Theorem 1] the result was formulated in terms of a holomorphic immersion $f: S \ \hookrightarrow X$.  The proofs for submanifolds $S \subset X$ and for immersions are identical. The proof of 
[16, Theorem 1] essentially relies on showing that from the second fundamental form $\sigma_{S|X}$ one obtains the harmonic representative for the cohomology class representing the obstruction to the holomorphic splitting of the holomorphic tangent sequence $0 \to T(S) \to T(X)|_S \to N_{S|X} \to 0$. The proofs are the same in the cases where $X$ is a complex hyperbolic form, a complex Euclidean form, or the complex projective space equipped with the Fubini-Study metric.

In this article our interest lies on compact splitting complex submanifolds of quotient manifolds of irreducible bounded symmetric domains in general.  As in the case of [16, Theorem 1] we will make use of the duality between bounded symmetric domains and their dual Hermitian symmetric spaces of the compact type. First of all we examine the minimal rational curves on these dual spaces which are uniruled projective manifolds.

Let $Z$ be an irreducible Hermitian symmetric space of the compact type. Denote by ${\cal O}(1)$ the positive generator of Pic$(Z) \cong \mathbb Z$.  Then, the vector space $\Gamma(Z,{\cal O}(1))$ of global holomorphic sections of ${\cal O}(1)$ defines a holomorphic embedding $\iota: Z \hookrightarrow \mathbb P\big(\Gamma(Z,{\cal O}(1))^*\big) := \mathbb P^N$, called the minimal embedding of $Z$. The projective submanifold $Z \subset \mathbb P^N$ is uniruled by projective lines.  For a point $x \in Z$, by the variety of minimal rational tangents (VMRT) ${\mathscr C}_x(Z) \subset \mathbb PT_x(Z)$ we mean the set consisting of all projectivizations $[\alpha]$ of non-zero tangent vectors $\alpha \in T_x(Z)$ such that $T_x(\ell) = \mathbb C\alpha$ for some projective line $\ell \subset Z$ passing through $x$. Varying $x \in Z$ we obtain the VMRT structure $\pi_Z: {\mathscr C} (Z) \to Z$ over $Z$.  The action of Aut$(Z)$ on $Z$ induces an action of Aut$(Z)$ on ${\mathscr C} (Z)$ respecting the projection $\pi_Z: {\mathscr C} (Z) \to Z$.  Let $\Omega \subset Z$ be the Borel embedding.  Then the identity component of Aut$(\Omega)$ is naturally identified with the subgroup of Aut$(Z)$ which preserves $\Omega$.
We write ${\mathscr C} (\Omega) := {\mathscr C} (Z)|_{\Omega}$, define $\pi_0 := \pi|_\Omega$, and call $\pi_0: {\mathscr C} (\Omega) \to \Omega$ the restricted VMRT structure on $\Omega$.  Let now $\Gamma \subset \text{\rm Aut}(\Omega)$ be a torsion-free discrete subgroup, and write $X = \Omega/\Gamma$.  The action of $\Gamma$ on $\Omega$ lifts to an action of $\Gamma$ on $\mathbb PT(\Omega)$, and it preserves the restricted VMRT structure $\pi_0: {\mathscr C} (\Omega) \to \Omega$. Thus, the VMRT structure on $\Omega$ descends to a locally homogeneous holomorphic fiber bundle $\pi_X: {\mathscr C} (X) \to X$ where each fiber ${\mathscr C}_x(X)$ is identified by lifting to $\Omega$ and by applying an automorphism of $\Omega$ to the reference VMRT ${\mathscr C}_0(X) \subset \mathbb PT_0(\Omega) = \mathbb PT_0(Z)$, where $\Omega \Subset \mathbb C^n \subset Z$ are the natural inclusions of $\Omega$ incorporating the Harish-Chandra realization $\Omega \Subset \mathbb C^n$ and the Borel embedding $\Omega \subset Z$, and $0 \in \Omega$ is the origin of $\mathbb C^n$.  We call $\pi_X: {\mathscr C} (X) \to X$ the quotient (restricted) VMRT structure on $X$.

Concerning characteristic complex submanifolds $S$ of $X = \Omega/\Gamma$ we need the following basic definitions.

\begin{defi} [\textbf{characteristic vector, characteristic $s$-plane}]
On $X = \Omega/\Gamma$ denote by $\pi_X: {\mathscr C} (X) \to X$ the quotient VMRT structure on $X$. A tangent vector $\alpha\in T(X)$ such that $[\alpha]\in{\mathscr C} (X)$ is called a characteristic vector.
Let $s > 0$ be an integer and $x$ be a point on $X$.  A member $\Pi$ of the Grassmannian $\text{\rm Gr}(s,T_x(X))$ of $s$-planes in $T_x(X)$ is called a characteristic $s$-plane if and only if $\,\mathbb P\Pi \subset {\mathscr C}_x(X)$.
In general, for any $s>0$ we call $\Pi$ a characteristic vector subspace.
\end{defi}

\begin{defi} [{\bf characteristic complex submanifold}]
A complex submanifold $S \subset X=\Omega/\Gamma$, $s := \dim(S)$, is said to be a characteristic complex submanifold if and only if the holomorphic tangent subspace $T_x(S) \subset T_x(X)$ is a characteristic $s$-plane for any point $x \in S$, i.e., if and only if \, $\mathbb PT_x(S) \subset {\mathscr C}_x(X)$ for any $x \in S$.
\end{defi}

We prove first of all the following characterization theorem for compact characteristic complex submanifolds $S \subset X = \Omega/\Gamma$.

\setcounter{theorem}{0}
\begin{theorem}
Let $\, \Omega$ be an irreducible bounded symmetric domain and $g_\Omega$ be a canonical K\"ahler-Einstein metric on $\Omega$.  Let $\, \Gamma \subset \text{\rm Aut}(\Omega)$ be a torsion-free discrete subgroup, write $X = \Omega/\Gamma$, and let $g$ be the K\"ahler-Einstein metric on $X$ induced from $g_\Omega$.
Let $S \subset X$ be a compact characteristic complex submanifold.  Assume that $S$ is a splitting complex submanifold in $X$.  Then $(S,g|_S) \hookrightarrow (X,g)$ is totally geodesic. Moreover, writing $T(X)|_S = T(S) \oplus {\cal N}$ for a holomorphic splitting of the tangent sequence over $S$, we have necessarily ${\cal N} = T(S)^\perp$.
\end{theorem}

\subsection{Characterization of compact totally geodesic submanifolds of quotients of bounded symmetric domains by means of the existence of K\"ahler-Einstein metrics} \quad
Regarding the use of K\"ahler-Einstein metrics, the starting point is the well-known existence theorem of Aubin and Yau.


\begin{nonnumber} {\rm (See \cite{s1}, \cite{s25}).}
Let $M$ be a compact K\"ahler manifold with ample canonical line bundle.  Then, there exists on $M$ a K\"ahler-Einstein metric of negative Ricci curvature. Moreover, such metrics are unique up to scaling factors.
\end{nonnumber}

By a canonical K\"ahler-Einstein metric on $M$ we will mean any of the K\"ahler-Einstein metrics on $M$ of Ricci curvature $-c$, where $c > 0$. The following existence result for compact complex submanifolds of quotients of bounded symmetric domains is well-known, but we include here a proof for easy reference.

\begin{prop}
Let $\Omega$ be a bounded symmetric domain and $\, \Gamma \subset \text{\rm Aut}(\Omega)$ be a torsion-free discrete subgroup, and write $X = \Omega/\Gamma$. Let $S \subset X$ be a compact complex submanifold.  Then, for any $c > 0$ there exists a unique K\"ahler-Einstein metric on $S$ of Ricci curvature $-c$.
\end{prop}

\begin{proof}
If $\Omega$ is irreducible, $g_\Omega$ agrees up to a scaling factor with the invariant K\"ahler metric defined by the Killing form, so that $(\Omega,g_\Omega)$ is of nonpositive holomorphic bisectional curvature and of strictly negative holomorphic sectional curvature, and the same holds when $\Omega$ is reducible by de Rham decomposition. By invariance $g_\Omega$ descends to a K\"ahler-Einstein metric $g$ on $X = \Omega/\Gamma$. We denote by $R$ the curvature tensor of $(X,g)$ and by $R^S$ the curvature tensor of $(S,g|_S)$. Let $x \in S$ and $\xi, \eta \in T_x(S)$.  By the Gauss equation, we have
$R^S_{\xi\ol{\xi}\eta\ol{\eta}} = R_{\xi\ol{\xi}\eta\ol{\eta}} - \|\sigma(\xi,\eta)\|^2 \le 0$,
where $\sigma = \sigma_{S|X}$ denotes the second fundamental form of $(S,g|_S) \hookrightarrow (X,g)$. If $\xi = \eta \neq 0$, then $R^S_{\xi\ol{\xi}\xi\ol{\xi}} = R_{\xi\ol{\xi}\xi\ol{\xi}} - \|\sigma(\xi,\xi)\|^2 < 0$, and it follows that $(S,g|_S)$ is of (strictly) negative Ricci curvature.
By the theorem of Aubin-Yau, for any $c > 0$ there exists a unique K\"ahler-Einstein metric of Ricci curvature $-c$ on $S$. 
\end{proof}

\noindent\textbf{Remark.}
Here and henceforth $\|\cdot\|$ denotes the norm of a vector measured against a Hermitian metric which is implicitly understood in the given context.

\vskip 0.2cm
On a complex manifold $M$ we denote by $T^r_s(M):=T(M)^{\otimes r}\otimes T^*(M)^{\otimes s}$ the holomorphic bundle of $(r,s)$-tensors, i.e., tensors which are contravariant of degree $r$ and covariant of degree $s$.
Let $(\Omega,g_\Omega)$ be an irreducible bounded symmetric domain equipped with a K\"ahler-Einstein metric $g_\Omega$, and let $X=\Omega/\Gamma$, where $\Gamma\subset\text{\rm Aut}(\Omega)$ is a torsion-free discrete subgroup. Denote by $R_{\xi\bar\eta\mu\bar\nu}$ the curvature tensor of $X$.  By contracting with the metric, we obtain from $R_{\xi\bar\eta\mu\bar\nu}$ a $(2,2)$-tensor $\displaystyle R^{\alpha\beta}_{\gamma\delta}:=\sum_{\eta,\nu}R_{\gamma\bar\eta\delta\bar\nu}g^{\alpha\bar\eta}g^{\beta\bar\nu}$. As $X$ is Hermitian locally symmetric, $R^{\alpha\beta}_{\gamma\delta}$ is parallel and hence holomorphic, i.e. $R^{\alpha\beta}_{\gamma\delta}\in H^0(X,T^2_2(X)$). In particular, we can regard $R^{\alpha\beta}_{\gamma\delta}$ as a holomorphic endomorphism on $T^1_1(X)$. In addition, due to the symmetries of the curvature tensor, we have $R^{\alpha\beta}_{\gamma\delta}=R^{\beta\alpha}_{\gamma\delta}=R^{\alpha\beta}_{\delta\gamma}$ and hence $R^{\alpha\beta}_{\gamma\delta}$ can also be regarded as a holomorphic endomorphism on $S^2T(X)$. To avoid confusion, we denote the two endomorphisms respectively by $R_\tau:T^1_1(X)\rightarrow T^1_1(X)$ and $R_\sigma:S^2T(X)\rightarrow S^2T(X)$.

Since $R_\sigma$ is parallel, $S^2T(X)$ has a parallel (hence holomorphic) direct-sum decomposition in which the direct summands are the eigenbundles of $R_\sigma$. For the purpose of obtaining vanishing theorems for cohomology groups of $X=\Omega/\Gamma$, Calabi-Vesentini and Borel computed this decomposition for all irreducible bounded symmetric domains $\Omega$ (Calabi-Vesentini for the classical types and Borel for the exceptional types). They showed that when $\text{\rm rank}(\Omega)\geq 2$, the endomorphism $R_\sigma$ always has exactly two eigenvalues and hence there is a two-factor parallel direct-sum decomposition $S^2T(X)=A\oplus B$.

\vskip 0.2cm
Finally, we recall the following notion.

\begin{defi} [{\bf degree of the strong non-degeneracy of the bisectional curvature}] 
Let $X = \Omega/\Gamma$, where $\Omega$ is an irreducible bounded symmetric domain and $\Gamma\subset\text{\rm Aut}(\Omega)$ is a torsion-free discrete subgroup. Let $p\in X$ and ${\cal Z}_p=\{(A,B): A\subset T_p(X), B\subset T_p(X)$ are linear subspaces such that $R_{a\bar a b\bar b}=0$ for all $(a,b)\in A\times B\}$. Then $\displaystyle\max_{(A,B)\in{\cal Z}_p}\{\dim A+\dim B\}$, which is independent of $p$, is called the degree of the strong non-degeneracy of the bisectional curvature of $X$~\cite{s20}.
\end{defi}

Exploiting the existence of K\"ahler-Einstein metrics on compact complex submanifolds of quotient manifolds $X$ of irreducible bounded symmetric domains $\Omega$ as given in Proposition 2.1, we show in the current article that compact splitting complex submanifolds of $X$ of sufficiently small codimension (in some specific sense depending on $\Omega$) are necessarily totally geodesic, as given in the ensuing Theorem 2.2 and Theorem 2.3.

\begin{theorem}
Let $\Omega$ be an irreducible bounded symmetric domain and $\text{\rm rank}(\Omega)\geq 2$. Let $\Gamma\subset\text{\rm Aut}(\Omega)$ be a torsion-free discrete subgroup and write $X = \Omega/\Gamma$. Write the eigenbundle decomposition $S^2T(X)=A\oplus B$ for $R_\sigma$, where $R_\sigma$ is the endomorphism on $S^2T(X)$ induced by the curvature tensor. Let $S\subset X$ be a compact splitting complex submanifold such that $\text{\rm rank}(S^2T(S))>\max\{\text{\rm rank}(A),\text{\rm rank}(B)\}$ and $\text{\rm dim}(S)>\rho$, where $\rho$ is the degree of the strong non-degeneracy of the bisectional curvature of $X$. Then, $S$ is Hermitian locally symmetric of rank at least 2 and totally geodesic with respect to the canonical K\"ahler-Einstein metrics on $X$.
\end{theorem}

\begin{theorem}
Let $\Omega$ be an irreducible bounded symmetric domain with $\text{\rm rank}(\Omega)\geq 2$, $K\subset\text{\rm Aut}(\Omega)$ be the isotropy group at $0\in\Omega$ and $\frak k$ be its Lie algebra. Let $\Gamma\subset\text{\rm Aut}(\Omega)$ be a torsion-free discrete subgroup and $X=\Omega/\Gamma$. Write the eigenbundle decomposition $S^2T(X)=A\oplus B$ for $R_\sigma$, where $R_\sigma$ is the endomorphism on $S^2T(X)$ induced by the curvature tensor. Let $S\subset X$ be a compact splitting complex submanifold such that $\text{\rm rank}(S^2T(S))>\min\{\text{\rm rank}(A),\text{\rm rank}(B)\}$ and $\text{\rm dim}(S)^2 > \text{\rm max}\{\text{\rm dim}_{\mathbb R}(\frak k),\rho^2\}$, where $\rho$ is the degree of the strong non-degeneracy of the bisectional curvature of $X$. Then, $S$ is Hermitian locally symmetric of rank at least 2 and totally geodesic with respect to the canonical K\"ahler-Einstein metrics on $X$.
\end{theorem}


\section{Compact characteristic submanifolds on quotients of irreducible bounded symmetric domains}

Any bounded symmetric domain $\Omega$ admits an invariant complete K\"ahler-Einstein metric of negative Ricci curvature.  By the normalized canonical K\"ahler-Einstein metric $g_\Omega$ on $\Omega$ we will mean that $g_\Omega$ is chosen such that minimal disks on $\Omega$ are of Gaussian curvature $-2$, noting that all minimal disks are equivalent to each other under Aut$(\Omega)$. For any torsion-free discrete subgroup $\Gamma \subset \text{\rm Aut}(\Omega)$ and for $X := \Omega/\Gamma$, the K\"ahler-Einstein metric $g_\Omega$ on $\Omega$ descends by $\Gamma$-invariance to a canonical K\"ahler-Einstein metric $g$ on $X$.

For curvature estimates on bounded symmetric domains we will make use of dual pairs of Hermitian symmetric spaces. Let $\Omega \Subset \mathbb C^n \subset Z$ be the standard embeddings incorporating the Harish-Chandra realization $\Omega \Subset \mathbb C^n$ and the Borel embedding $\Omega \subset Z$.  
Write $G = \text{\rm Aut}_0(\Omega)$, the identity component of the group of biholomorphic automorphisms of $\Omega$, so that also $G = \text{\rm Aut}_0(\Omega,g_\Omega)$. Write $\frak g_\Omega = \frak m \oplus \frak k$ for the Cartan decomposition with respect to the Cartan involution of $(\Omega,g_\Omega)$ at $0 \in \Omega \Subset \mathbb C^n$, where $\frak g_\Omega$ resp.\,$\frak k$ stands for the Lie algebras of $G$ resp.\,$K$, and $\frak m$ is canonically identified with the real tangent space at $0$.  Let $K = \text{\rm Aut}_0(\Omega;0)\subset G$ be the isotropy subgroup at $0$, so that also $K = \text{\rm Aut}_0(\Omega,g_\Omega;0)$ is the isotropy subgroup at 0 of the group of biholomorphic isometries of $\Omega$.  The identity component $\text{\rm Aut}_0(Z)$ of the
automorphism group $\text{\rm Aut}(Z)$ is a complexification of $G$, and we will write $\text{\rm Aut}_0(Z) = G^\mathbb C$.  There is a unique compact real form $G_c$ of $G^\mathbb C$ such that, writing $\frak g_c$ resp.\,$\frak g^\mathbb C$ for the Lie algebras of $G_c$ resp.\,$G^\mathbb C$, we have $\frak g_c = \sqrt{-1}\frak m \oplus \frak k \subset \frak g_\Omega \otimes_{\mathbb R} \mathbb C = \frak g^\mathbb C$. There is a unique $G_c$-invariant K\"ahler-Einstein metric $g_c$ on $Z$ such that $g_c$ agrees with $g_\Omega$ at $0 \in \Omega$. It defines the structure of a Hermitian symmetric space on $Z$, and $\big((\Omega,g_\Omega);(Z,g_c)\big)$ is a dual pair of Riemannian symmetric spaces (see 
\cite{s9}).  We have

\begin{lemma}
Let $\Omega$ be an irreducible bounded symmetric domain, $\text{\rm dim}(\Omega) =: n$, $g_\Omega$ be the normalized canonical K\"ahler-Einstein metric on $\Omega$, and denote by $R$ the curvature tensor of $(\Omega, g_\Omega)$. Let $x$ be any point on $\Omega$, $\chi \in T_x(\Omega)$ be a unit vector, and denote by $H_\chi$ the Hermitian bilinear form on $T_x(\Omega)$ defined by $H_\chi(\xi,\eta) := R_{{\chi}\ol{\chi}\xi\ol{\eta}}$. Let $s$ be a positive integer, $1 \le s \le n$, and let $\Pi \subset T_x(\Omega)$ be an $s$-dimensional complex vector subspace. Then, $\text{\rm Tr}_{g_\Omega}\big(H_\chi|_\Pi\big) \ge -(s+1)$. Moreover, if equality is attained for every unit vector $\chi$ of an $s$-dimensional complex vector subspace $\Pi' \subset T_x(\Omega)$, then $\Pi \subset T_x(\Omega)$ is a characteristic $s$-plane and $\Pi' = \Pi$.
\end{lemma}

\begin{proof}
Denote by $W$ the curvature tensor of $(Z,g_c)$.  As mentioned $g_\Omega$ agrees with $g_c$ at $0 \in \Omega \subset Z$.
Moreover, the curvature tensors of $(\Omega,g_\Omega)$ and $(Z,g_c)$ are opposite to each other at $0$ in the sense that, for any ordered quadruple $(\xi,\eta,\mu,\nu)$ of tangent vectors at $0$ of type (1,0) we have
$$
R_{{\xi}\ol{\eta}\mu\ol{\nu}} = -W_{{\xi}\ol{\eta}\mu\ol{\nu}} \, . \eqno{(1)}
$$
To prove Lemma 3.1 it suffices therefore to consider $(Z,g_c)$. The group $G^{\mathbb C}$ acts on ${\cal O}(1)$ and hence on $\Gamma(Z,{\cal O}(1))$.  Equipping the latter space with a $G_c$-invariant metric and hence endowing $\mathbb P(\Gamma(Z,{\cal O}(1)^*) \cong \mathbb P^N$ with the unique Fubini-Study metric $ds_{FS}^2$ of constant holomorphic sectional curvature $+2$ invariant under the action of $G_c$, the minimal embedding $\iota: Z \hookrightarrow \mathbb P^N$ is a holomorphic isometric embedding of $(Z,g_c)$ into $(\mathbb P^N,ds_{FS}^2)$. Denote by $F$ the curvature tensor of $(\mathbb P^N,ds_{FS}^2)$, and by $\tau$ the second fundamental form of the isometric embedding $\iota: (Z,g_c) \hookrightarrow (\mathbb P^N,ds_{FS}^2)$. At $0 \in Z$, identified with $\iota(0) \in \mathbb P^N$, for $\chi, \xi \in T_0(Z)$ by the Gauss equation we have
$$
W_{{\chi}\ol{\chi}\xi\ol{\xi}} = F_{{\chi}\ol{\chi}\xi\ol{\xi}} - \|\tau(\chi,\xi)\|^2
\le F_{{\chi}\ol{\chi}\xi\ol{\xi}}\, . \eqno{(2)}
$$
Note that
$$
F_{i\ol{j}k\ol{\ell}} = \delta_{ij}\delta_{k\ell} + \delta_{i\ell}\delta_{jk} \, \eqno{(3)}
$$
in terms of an orthonormal basis $\{e_1,\cdots,e_n\}$ of $T_0(Z)$ with respect to $g_c$.  Note also that for a unit vector $\chi \in T_0(\mathbb P^N)$ we have $F_{{\chi}\ol{\chi}\chi\ol{\chi}} = 2$.  Furthermore, if $\xi \in T_0(\mathbb P^N)$ is a unit vector orthogonal to $\chi$, we have $F_{{\chi}\ol{\chi}\xi\ol{\xi}} = 1$.  From now on $\chi$ denotes a unit vector in $T_0(Z)$. It follows from (2) that $W_{\chi\ol{\chi}\chi\ol{\chi}} = F_{\chi\ol{\chi}\chi\ol{\chi}} = 2$ if and only if $\tau(\chi,\chi) = 0$.   In other words, $\chi$ is a characteristic vector if and only if $\tau(\chi,\chi) = 0$.  Let $\Pi \in \text{\rm Gr}(s,T_0(Z))$.  Decompose $\chi = \mu + \nu$ where $\mu \in \Pi$ and $\nu \perp \Pi$. Let $\{\xi_1,\xi_2,\cdots,\xi_s\}$ be an orthonormal basis of $\Pi$ such that $\mu \in \mathbb C\xi_1$.  Note from (3) that 
$F_{{\mu}\ol{\nu}\xi\ol{\xi'}} = 0$ for any $\xi, \xi' \in \Pi$. 
Let $H^c_\chi$ be the Hermitian form on $T_0(Z)$ defined by $H^c_\chi(\xi,\eta) = W_{\chi\ol{\chi}\xi\ol{\eta}}$.  We compute now
$$
\text{\rm Tr}_{g_c}\big(H^c_\chi|_\Pi\big) = \sum_{i=1}^s W_{{\chi}\ol{\chi}\xi_i\ol{\xi_i}} \le \sum_{i=1}^s F_{{\chi}\ol{\chi}\xi_i\ol{\xi_i}}
= \sum_{i=1}^s F_{{\mu}\ol{\mu}\xi_i\ol{\xi_i}} + \sum_{i=1}^s F_{{\nu}\ol{\nu}\xi_i\ol{\xi_i}} \, , \eqno{(4)}
$$
recalling that $F_{{\mu}\ol{\nu}\xi\ol{\xi'}} = 0$ for any $\xi, \xi' \in \Pi$. It follows that
{\small
$$
\text{\rm Tr}_{g_\Omega}\big(H_\chi|_\Pi\big) = -\text{\rm Tr}_{g_c}\big(H^c_\chi|_\Pi\big)
\ge  - (s+1)\|\mu\|^2 -s\|\nu\|^2 \ge -(s+1)(\|\nu\|^2+\|\mu\|^2) = -(s+1) \, . \eqno{(5)}
$$
}
From the intermediate inequalities in (4) and (5) it follows that equality holds if and only if $\nu = 0$ and $\tau(\chi,\xi)$ = 0 for any $\xi \in \Pi$, i.e., if and only if $\chi \in \Pi$ and $\tau(\chi,\xi) = 0$ for any $\xi \in \Pi$.  If $\Pi \in \text{\rm Gr}(s,T_0(\Omega))$ is such that equality holds in (5) for every unit vector $\chi \in \Pi'$ it follows that $\Pi' = \Pi$ and $\tau(\xi,\eta) = 0$ for and $\xi, \eta \in \Pi$.  In particular $\tau(\xi,\xi) = 0$ for any $\xi \in \Pi$, so that
$\Pi \subset T_0(\Omega) = T_0(Z)$ is a characteristic $s$-plane, as desired.
\end{proof}

\vspace{1.2mm}
\noindent
{\it Proof of Theorem 2.1.}
By hypothesis the tangent sequence  $0 \to T(S) \to T(X)|_S \to N_{S|X} \to 0$ splits holomorphically over $S$, $T(X)|_S = T(S) \oplus {\cal N}$, where ${\cal N} \subset T(X)|_S$ is a lifting of $N_{S|X}$ to $T(X)|_S$. Let $R$ denote the curvature tensor of $(X,g)$ and $\sigma$ denote the second fundamental form of $(S,g|_S) \ \hookrightarrow (X,g)$. Computing the curvature tensor $R^S$ of $(S,g|_S)$, for any point $x \in S$, and any pair of vectors $\alpha, \beta \in T_x(S)$, by the Gauss equation we have
$$
R^S_{\alpha\ol{\alpha}\beta\ol{\beta}} = R_{\alpha\ol{\alpha}\beta\ol{\beta}} - \|\sigma(\alpha,\beta)\|^2 \, . \eqno{(1)}
$$
Let $\{\beta_1, \cdots, \beta_s\}$, $s = \dim(S)$ be an orthonormal basis of $T_x(S)$.  Denoting by $Ric^S$ the Ricci tensor of $(S,g|_S)$, for any $\alpha \in T_x(S)$ we have
$$
Ric^S_{\alpha\ol{\alpha}} = \sum_{i=1}^{s}R^S_{\alpha\ol{\alpha}\beta_i\ol{\beta_i}}
= \sum_{i=1}^{s}R_{\alpha\ol{\alpha}\beta_i\ol{\beta_i}} - \sum_{i=1}^{s}\|\sigma(\alpha,\beta_i)\|^2 \, . \eqno{(2)}
$$
By hypothesis $\mathbb PT_x(S) \subset {\mathscr C}_x(X)$ for every point $x \in X$. By the standard calculation implicit in (2) and (3) in the proof of Lemma 3.1 we have $\sum_{i=1}^{s}R_{\alpha\ol{\alpha}\beta_i\ol{\beta_i}} = -(s+1)\|\alpha\|^2$. We also write $\rho := \sqrt{-1}\sum_{i=1}^{s}Ric^S_{i\ol{j}}dz^i\wedge d\ol{z^j}$ for the Ricci form of $(S,g|_S)$, $\rho_{i\ol{j}} := Ric^S_{i\ol{j}}$. Hence we have
$$
\rho_{\alpha\ol{\alpha}} = -(s+1)\|\alpha\|^2 - \sum_{i=1}^{s}\|\sigma(\alpha,\beta_i)\|^2 \le -(s+1)\|\alpha\|^2\, . \eqno{(3)}
$$
We can make use of the holomorphic splitting to consider $T(S)$ as a quotient bundle of $T(X)|_S$.  In other words, we consider the
short exact sequence
$$
(\sharp)\,\,\, 0 \to {\cal N} \to T(X)|_S \overset{\epsilon}\longrightarrow{{\cal Q}}: = T(X)|_S/{\cal N} \to 0,
$$
${\cal Q} \cong T(S)$ as holomorphic vector bundles.  Let $h$ be the quotient Hermitian metric on ${\cal Q}$ induced by the Hermitian metric $g$ on $T(X)$ from the short exact sequence and denote by $\Theta$ the curvature tensor of $({\cal Q},h)$.  For any $x \in S$ and any two vectors $\xi, \eta \in {\cal Q}_x$, we have $h(\xi,\eta) = g(\xi',\eta')$ where $\xi'$ and $\eta'$ belong to the orthogonal complement ${\cal N}_x^\perp$ of ${\cal N}_x$ in $T_x(X)$ and they are uniquely determined by $\epsilon(\xi') = \xi, \epsilon(\eta') = \eta$. For $\xi \in {\cal Q}_x$ and $\alpha \in T_x(S)$ we have the curvature formula
$$
\Theta_{\xi\ol{\xi}\alpha\ol{\alpha}} = R_{\xi'\ol{\xi'}\alpha\ol{\alpha}} + \|\lambda(\alpha,\overline{\xi'})\|^2 \ge R_{\xi'\ol{\xi'}\alpha\ol{\alpha}} \, ,\eqno{(4)}
$$
for the quotient bundle $({\cal Q},h)$, where, denoting by $0 \to {\cal Q}^* \to T^*(X)|_S \to {\cal N}^* \to 0$ the dual of the short exact sequence $(\sharp)$,
we have $\lambda(\alpha,\overline{\xi'}) := \zeta(\alpha,\chi)$ for the second fundamental form $\zeta$ of $({\cal Q}^*,g^*|_{{\cal Q}^*}) \hookrightarrow (T^*(X)|_S,g^*)$, $\chi$ being the lifting of $\ol{\xi'}$ to ${\cal Q}^*_x$ by the Hermitian metric $g^*|_{{\cal Q}^*}$ (see~\cite{s8}). Here ${\cal Q}^* \subset T^*(X)|_S$ is the holomorphic subbundle of tangent covectors which annihilate $\cal N$, and we have canonically an isomorphism ${\cal Q}^* \cong T(S)^*$ as holomorphic vector bundles. Denote by $\theta
:= \sqrt{-1}\sum_{i=1}^{s}\theta_{i\ol{j}}dz^i\wedge d\ol{z^j}$ the curvature form of $\big(\det({\cal Q}),\det(h)\big)$.  We have
$$
\theta_{\alpha\ol{\alpha}} = \sum_{i=1}^{s}\Theta_{\beta_i\ol{\beta_i}\alpha\ol{\alpha}}
\ge \sum_{i=1}^{s}R_{\beta_i\ol{\beta_i}\alpha\ol{\alpha}} \eqno{(5)}
$$
for any orthonormal basis $\{\beta_1,\cdots,\beta_s\}$ of ${\cal N}_x^\perp$.  By (4), (5) and Lemma 3.1 we have
$$
\theta_{\alpha\ol{\alpha}} \ge -(s+1)\|\alpha\|^2 \, .\eqno{(6)}
$$
From (3) and (6),
we have
$$
\rho_{\alpha\ol{\alpha}} \le \theta_{\alpha\ol{\alpha}} - \sum_{i=1}^{s}\|\sigma(\alpha,\beta_i)\|^2 \, .\eqno{(7)}
$$
On the other hand, $\rho$ and $\theta$ are closed smooth (1,1)-forms on $S$ representing $2\pi$ times the first Chern class of $T(S)$ over $S$.  Denoting by $\omega_g$ the K\"ahler form of $(X,g)$ we have
$$
0 = \int_S (\rho-\theta) \wedge \omega_g^{s-1} \le 0 \, , \eqno{(8)}
$$
which forces by (7) that the trace of $\rho-\theta\le 0$ vanishes identically on $S$, hence $\rho \equiv \theta$ and $\sigma(\alpha,\beta) = 0$ for any $x \in S$ and for any $\alpha,\beta \in T_x(S)$. In other words, $(S,g|_S) \hookrightarrow (X,g)$ is totally geodesic.  Moreover, from (6) and (8) it follows that $\rho_{\alpha\ol{\alpha}} = \theta_{\alpha\ol{\alpha}} = -(s+1)\|\alpha\|^2$ for every tangent vector $\alpha \in T_x(S)$, and by Lemma (3.1) we conclude that ${\cal N}_x^\perp \subset T_x(X)$ is a characteristic $2$-plane and that furthermore $T_x(S) = {\cal N}_x^\perp$.  It follows that ${\cal N} = T(S)^\perp$, and the proof of Theorem 2.1 is complete. \quad $\square$

The proof given here works also to give a proof of  [16, Theorem 1] in the case where the ambient manifold $(X,g)$ is a flat Euclidean space, but it fails when $(X,g)$ is the complex projective space endowed with the Fubini-Study metric. One can prove the analogue of Theorem 2.1 when $(X,g)$ is replaced by an irreducible Hermitian symmetric space $(Z,g)$ of the compact type by adapting the proof of Theorem 1 in~\cite{s16} and showing that for a compact characteristic complex submanifold $(S,g|_S) \ \hookrightarrow (Z,g)$ the $(N^*_{S|Z} \otimes T_S)$-valued (0,1)-form $\mu$ derived from the second fundamental form $\sigma_{S|Z}$ remains $\overline{\partial}^*$-closed. This is the case because $\sigma_{S|Z}$ remains holomorphic whenever $\mathbb PT(S) \subset {\mathscr C} (Z)$, since $\partial_{\overline{\chi}}\sigma_{\beta\gamma} = -R_{{\beta}\overline{\chi}\gamma}\ \text{\rm mod} \ T_x(S)$, and the R.H.S. vanishes because

\begin{enumerate}
\item[$(\dag)$] $R_{{\beta}\overline{\chi}\gamma\overline{\eta}} = 0$ whenever $T_x(S) \subset T_x(Z)$ consists of characteristic vectors and $\eta$ is orthogonal to $T_x(S)$.
\end{enumerate}

When $Z \cong \mathbb P^n$, $(\dag)$ follows from the curvature formula for the Fubini-Study metric
(see (3) in the proof of Lemma 3.1). In general $(\dag)$ follows from the fact that, viewing the curvature tensor $R_x$ at $x \in S$ as a Hermitian bilinear form $P_x$ on $S^2T_x(Z)$, the symmetric square $\alpha\odot \alpha$ of any characteristic vector $\alpha$ is an eigenvector of $P$, giving $R_{{\alpha}\overline{\alpha}\gamma\overline{\alpha}} = 0$ whenever $\gamma$ is orthogonal to $\alpha$.
Then $(\dag)$ follows by polarization whenever $T_x(S) \subset T_x(Z)$ is a characteristic $s$-plane, $s = \dim(S)$. (Alternatively, $(\dag)$ follows from the case of the Fubini-Study metric and the formula for the curvature tensor $\big(W_{{\beta}\overline{\chi}\gamma\overline{\eta}}\big)$ of  $(Z,g_c)$ which follows from (2) in the proof of Lemma 3.1 by polarization.)

\vskip 0.2cm
In this article we are only concerned with compact complex submanifolds of quotients of bounded symmetric domains, for which the new and more direct proof given here for Theorem 2.1 suffices.

\section{Splitting complex submanifolds of quotients of irreducible bounded symmetric domains: splitting of tensor bundles on $\Omega$ versus splitting of tangent sequence on $S$}

Let $0\rightarrow V\overset{\iota}{\rightarrow} U\overset{p}{\rightarrow} W\rightarrow 0$ be an exact sequence of holomorphic vector bundles over a complex manifold $M$. The sequence splits holomorphically if and only if there exists a holomorphic bundle map $\pi:U\rightarrow V$ such that $\pi\circ\iota =\text{\rm id}_V$. (The kernel of $\pi$ then gives a holomorphic complement of $V$ in $U$.) Furthermore, it is also equivalent to the existence of a holomorphic bundle map $q:W\rightarrow U$ such that $p\circ q=\text{\rm id}_W$. (It is now the image of $q$ which gives a holomorphic complement of $V$ in $U$.)

Thus, for a splitting complex submanifold $S\subset X:=\Omega/\Gamma$, from the splitting of the holomorphic tangent sequence $$0\rightarrow T(S)\overset{\iota}{\longrightarrow} T(X)|_S\overset{p}{\longrightarrow} N_{S|X}\rightarrow 0,$$ we get a holomorphic bundle map $\pi:T(X)|_S\rightarrow T(S)$ with $\pi(v)=v$ for every $v\in T(S)$, where we have identified $T(S)$ as a subbundle of $T(X)|_S$. The projection $\pi$ naturally induces projection maps on various tensor bundles constructed from the tangent bundles of $S$ and $X$, and thus gives the splitting of the exact sequences associated to these tensor bundles, as follows.

Let $T^*(S)$ (resp.\,$T^*(X)$) be the holomorphic cotangent bundle of $S$ (resp.\,$X$). Consider the dual of the tangent sequence (i.e. the cotangent sequence), we write
$$
    0\rightarrow N^*_{S|X}\overset{i}{\longrightarrow} T^*(X)|_S\overset{\Pi}{\longrightarrow} T^*(S)\rightarrow 0.
$$
Now, given $\pi:T(X)|_S\rightarrow T(S)$ which splits the tangent sequence, we get a holomorphic bundle map $\pi^*:T^*(S)\rightarrow T^*(X)|_S$ defined by $\pi^*(v^*)=v^*\circ\pi$, where $v^*\in T^*(S)$. Note that the condition $\pi(v)=v$ for every $v\in T(S)$ implies that $\Pi\circ\pi^*=\text{\rm id}_{T^*(S)}$. Therefore, the cotangent sequence splits holomorphically.

On a complex manifold $M$, recall that $T^r_s(M):=T(M)^{\otimes r}\otimes T^*(M)^{\otimes s}$ denotes the holomorphic bundle of $(r,s)$-tensors. Define $\iota^r_s:=\iota^{\otimes r}\otimes \pi^{*\otimes s}:T^r_s(S)\rightarrow T^r_s(X)|_S$, then one gets an exact sequence
$$
    0\rightarrow T^r_s(S)\overset{\iota^r_s}{\longrightarrow} T^r_s(X)|_S\rightarrow N^r_s\rightarrow 0,
$$
where $N^r_s:=T^r_s(X)|_S/T^r_s(S)$. Note that $S$ being a splitting submanifold in $X$ is essential for us to embed $T^r_s(S)$ into $T^r_s(X)|_S$ when $s>0$.

\vskip 0.2cm
If we further define $\pi^r_s:=\pi^{\otimes r}\otimes \Pi^{\otimes s}:T^r_s(X)|_S\rightarrow T^r_s(S)$, then we have $\pi^r_s\circ\iota^r_s=\text{\rm id}_{T^r_s(S)}$. Thus, the above sequence of $(r,s)$-tensor bundles also splits holomorphically. It is clear that all the aforementioned bundle maps preserve the symmetric tensor products and hence similar conclusions hold for holomorphic bundles of symmetric tensors.

\vskip 0.2cm
In the rest of this section, we will only be considering the case for symmetric (2,0)-tensors
$$
    0\rightarrow S^2T(S)\mathop{\rightleftarrows}_{\pi_\sigma}^{\sigma} S^2T(X)|_S
$$
and the case where $(r,s)=(1,1)$,
$$
    0\rightarrow T^1_1(S)\mathop{\rightleftarrows}_{\pi_\tau}^{\tau}  T^1_1(X)|_S.
$$
Here, in order to simplify the notations, we write $\sigma:=\iota^2_0$, $\pi_\sigma:=\pi^2_0$, $\tau:=\iota^1_1$, $\pi_\tau:=\pi^1_1$ so that $\pi_\sigma\circ\sigma = \text{\rm id}_{S^2T(S)}$ and $\pi_\tau\circ\tau=\text{\rm id}_{T^1_1(S)}$.

\vskip 0.2cm
Now suppose $S\subset X$ is a compact splitting complex submanifold, where $X=\Omega/\Gamma$ with $\text{\rm rank}(\Omega)\geq 2$ and $\Omega$ is irreducible. We recall the following theorems by Calabi-Vesentini and Borel.

\begin{theorem}[{\rm See \cite{s3}, \cite{s5}}]
Let $\Omega$ be an irreducible bounded symmetric domain of rank $\ge 2$ and $\, \Gamma \subset \text{\rm Aut}(\Omega)$ be a torsion-free discrete subgroup, and write $X := \Omega/\Gamma$.
Then, the endomorphism $R_\sigma:S^2T(X)\rightarrow S^2T(X)$ has exactly two eigenvalues, which are non-zero, and there is a two-factor parallel direct-sum decomposition $$S^2T(X)=A\oplus B.$$
\end{theorem}


\begin{theorem} [{\rm See \cite{s3}}]
Let $\, \Omega$ be an irreducible bounded symmetric domain of rank $\ge 2$ and $\, \Gamma \subset \text{\rm Aut}(\Omega)$ be a torsion-free discrete subgroup, and write $X := \Omega/\Gamma$. Then, there is a parallel direct-sum decomposition
$$T^1_1(X)=C\oplus D,$$
where $C$ is the eigenbundle corresponding to the kernel of $R_\tau:T^1_1(X)\rightarrow T^1_1(X)$ and $\text{\rm rank}(C)>0$. If we write $G=\text{\rm Aut}(\Omega)$ and $K\subset G$ for the isotropy group at $0\in\Omega$, then $\text{\rm rank}(D)=\text{\rm dim}_{\mathbb R} \frak k$, where $\frak k$ is the Lie algebra of $K$.
\end{theorem}

Now consider
$$
\begin{matrix}
    S^2T(S)&\overset{\sigma}{\longrightarrow} &S^2T(X)|_S&\overset{\pi_\sigma}{\longrightarrow}& S^2T(S),\\
    &&\parallel&&\\
    &&A|_S\oplus B|_S&&
\end{matrix}
$$
where $\pi_\sigma\circ\sigma=\text{\rm id}_{S^2T(S)}$. From the decomposition $S^2T(X) = A \oplus B$ we define the canonical projections $p_A:S^2T(X)\rightarrow A$ and $p_B:S^2T(X)\rightarrow B$ and use them to decompose the identity map as
$$
    \text{\rm id}_{S^2T(S)} = E_A+ E_B,
$$
where $E_A=\pi_\sigma\circ p_A\circ\sigma$ and $E_B=\pi_\sigma\circ p_B\circ\sigma$ are both endomorphisms on $S^2T(S)$. We have the following observation.

\begin{prop}
The endomorphisms $E_A$, $E_B$ are parallel with respect to some K\"ahler-Einstein metric on $S$.
\end{prop}

\vspace{1.2mm}
\noindent
{\it Proof.}
From Proposition 2.1 we know that there exists a K\"ahler-Einstein metric on $S$. Since $E_A, E_B$ can be regarded as elements in $H^0(S, T^2_2(S))$, the proposition now follows from the following theorem of Bochner. \quad $\square$

\begin{theorem} [{\rm See \cite{s4}}]
Let $M$ be a compact K\"ahler-Einstein manifold, then all elements in $H^0(M, T^r_r(M))$ are parallel for every $r>0$.
\end{theorem}

The following classical theorems regarding the holonomy groups of irreducible Riemannian manifolds will be needed later.

\begin{theorem} [{\rm See \cite{s2}}]
Let $M$ be an irreducible K\"ahler manifold of dimension $n$, then either $M$ is Hermitian locally symmetric of rank at least 2, or the restricted holonomy group of $M$ is one of the following groups: $\text{\rm (i)}$ $U(n)$; $\text{\rm (ii)}$ $SU(n)$; $\text{\rm (iii)}$ $Sp(n/2)\times U(1)$; $\text{\rm (iv)}$ $Sp(n/2)$.
\end{theorem}

\begin{theorem} [{\rm See \cite{s23}}]
For $\text{\rm (i)}$ $U(n)$; $\text{\rm (ii)}$ $SU(n)$; $\text{\rm (iii)}$ $Sp(n/2)\times U(1)$; $\text{\rm (iv)}$ $Sp(n/2)$, their natural representations on the symmetric tensor product $S^m\mathbb C^n$ are irreducible for every $m$.
\end{theorem}

\vspace{1.2mm}
\noindent
{\bf Remark.} \ Theorems 4.4 and 4.5 in the present forms are taken from~\cite{s11}.

Combining Theorem 4.4 and Theorem 4.5, we have

\begin{theorem}
Let $M$ be an irreducible K\"ahler manifold. If there exists a non-trivial proper subspace in $S^2T(M)$ which is holonomy-invariant, then $M$ is Hermitian locally symmetric of rank at least 2.
\end{theorem}

\setcounter{lemma}{0}
\begin{lemma}
Let $X=\Omega/\Gamma$, where $\text{\rm rank}(\Omega)\geq 2$. Let $\rho$ be the the degree of the strong non-degeneracy of the bisectional curvature of $X$ {\rm (Definition 2.4)}. If $S\subset X$ is a compact complex submanifold such that $\text{\rm dim}(S)>\rho$, then $S$ is irreducible $\text{\rm (}$in the sense of de Rham$\text{\rm )}$ with respect to any K\"ahler metric on $S$.
\end{lemma}

\vspace{1.2mm}
\noindent
\begin{proof}
Let $h$ be a K\"ahler metric on $S$ such that $S$ is reducible.  Then after replacing $S$ by a finite unbranched cover if necessary, we can write $T(S)=E\oplus F$, where
$E$ and $F$ are parallel vector subbundles of $T(S)$, such that $e:=\text{\rm rank}(E)>0$ and $f:=\text{\rm rank}(F)>0$. Consider the curvature $(1,1)$-form $c_1(F, h)$ for $(F, h|_F)$ representing the first Chern class of $F$. Then its kernel contains $E$ and hence $c_1(F, h)^{f+1}\equiv 0$ on $S$. On the other hand, on $S$ we have the restriction of a canonical metric $g$ of $X$ and if we denote the curvature $(1,1)$-form for $(F, g|_F)$ by $c_1(F, g)$, which is cohomologous to $c_1(F, h)$, then
$$
    0=\int_S c_1(F, h)^{f+1}\wedge \omega_h^{e-1} = \int_S c_1(F, g)^{f+1}\wedge \omega_h^{e-1},
$$
where $\omega_h$ is the K\"ahler form of $h$. But $c_1(F, g)$ is seminegative as $g$ is of nonpositive holomorphic bisectional curvature. Thus, from the above integral formula we get $c_1(F, g)^{f+1}\equiv 0$ and hence at every point $p\in S$, $c_1(F, g)$ has a kernel $K_p\subset T_p(S)$ of dimension at least $e$. Now let $R^{S,g}$ be the curvature tensor of $(S,g|_S)$, then for $\mu\in K_p$, and $\nu\in F_p$, both of unit length, we have
$$
    0\geq R^{S,g}_{\mu\bar\mu\nu\bar\nu}\geq c_1(F, g)\left(\frac{1}{\sqrt{-1}}\,\mu\wedge\bar\mu\right)=0.
$$
Therefore, $R^{S,g}_{\mu\bar\mu\nu\bar\nu}=0$ and we deduce that
$$\rho\geq\dim(K_p)+\dim(F_p)\geq e+f=\dim (S).$$
\end{proof}

\vspace{1.2mm}
\noindent
{\bf Remark.} \ The degree of strong non-degeneracy of the bisectional curvature $\rho$ is defined for every K\"ahler manifold (for non-homogeneous cases, it varies from point to point). In the cases for $\Omega/\Gamma$, where $\Omega$ is irreducible and rank$(\Omega)\geq 2$, $\rho$ is calculated for all cases~\cite{s20} for classical cases and~\cite{s26} for exceptional cases):

$\,$

$D^I_{m,n}: \rho=(m-1)(n-1)+1$\,\,\,\,\,\,\,\,\,\,\,\,\,$D^{II}_{n}: \rho=\dfrac{(n-2)(n-3)}{2}+1$

$D^{III}_{n}: \rho=\dfrac{n(n-1)}{2}+1$ \,\,\,\,\,\,\,\,\,\,\,\,\,\,\,\,\,\,\,\,\,\,\,\,\,\,\, $D^{IV}_{n}: \rho=2$

$D^{V}: \rho=6$ \,\,\,\,\,\,\,\,\,\,\,\,\,\,\,\,\,\,\,\,\,\,\,\,\,\,\, \,\,\,\,\,\,\,\,\,\,\,\,\,\,\,\,\,\,\,\,\,\,\,\,\,\,\,\,\,\,\,\,\,\,$D^{VI}: \rho=11$.

$\,$

Here we note that $\rho$ is also the maximal dimension of the reducible (in the sense of de Rham) K\"ahler submanifolds of $\Omega/\Gamma$. On the one hand, if we have a reducible K\"ahler submanifold $S=M\times N\subset\Omega/\Gamma$, then for any $p\in M\times N$ and every $\mu\in T_p(M)$, $\nu\in T_p(N)$, we have $0=R^S_{\mu\bar\mu\nu\bar\nu}\leq R^{\Omega/\Gamma}_{\mu\bar\mu\nu\bar\nu}\leq 0$ $\Rightarrow$ $R^{\Omega/\Gamma}_{\mu\bar\mu\nu\bar\nu}=0$. Thus, $\dim(S)=\dim(M)+\dim(N)\leq\rho$. On the other hand, from the above table, one knows that there exists a reducible totally geodesic locally symmetric submanifold of dimension equal to $\rho$. They are listed as follows:

$\,$

$D^I_{m,n}: \mathbb B^1\times D^I_{m-1,n-1}$\,\,\,\,\,\,\,\,\,\,\,\,\,$D^{II}_{n}: \mathbb B^1\times D^{II}_{n-2}$

$D^{III}_{n}: \mathbb B^1\times D^{III}_{n-1}$ \,\,\,\,\,\,\,\,\,\,\,\,\,\,\,\,\,\,\,\,\,\,\,\, $D^{IV}_{n}: \mathbb B^1\times\mathbb B^1$

$D^{V}: \mathbb B^1\times\mathbb B^5$ \,\,\,\,\,\,\,\,\,\,\,\,\,\,\,\,\,\,\,\,\,\,\,\,\,\,\, \,\,\,\,\,\,\,$D^{VI}: \mathbb B^1\times D^{IV}_{10}$,

$\,$

\noindent 
where $\mathbb B^m$ is the $m$-dimensional complex unit ball.

$\,$

We are now ready to prove Theorem 2.2 and Theorem 2.3, as follows.

\vspace{1.2mm}
\noindent
\begin{proof}[Proof of Theorem 2.2.]
The assumption $\text{\rm rank}(S^2T(S))>\max\{\text{\rm rank}(A),\text{\rm rank}(B)\}$ implies that neither $p_A\circ\sigma:S^2T(S)\rightarrow A|_S$ nor $p_B\circ\sigma:S^2T(S)\rightarrow B|_S$ is injective. In particular, both $E_A=\pi_\sigma\circ p_A\circ\sigma$ and $E_B=\pi_\sigma\circ p_B\circ\sigma$ must have a non-trivial kernel at every point on $S$. Now from Proposition 4.1, $E_A$, $E_B$ are parallel with respect to some K\"ahler-Einstein metric $h$ on $S$. Hence, their kernels are invariant under the parallel transport with respect to the Riemannian connection $\nabla_h$ of $h$. If $E_A$ is not identically equal to the zero endomorphism, then at any point $p\in S$, we would get a non-trivial proper subspace of $S^2T_p(S)$ which is holonomy-invariant with respect to $\nabla_h$. Then, Lemma 4.1 together with Theorem 4.6 say that $S$ is Hermitian locally symmetric of rank at least 2. If $E_A$ is the zero endomorphism, then we can apply the same argument to $E_B$, which cannot be zero since $E_A+E_B=\text{\rm id}_{S^2T(S)}$. Thus, $S$ must be Hermitian locally symmetric of rank at least 2.

Since $X$ is Hermitian locally symmetric of non-compact type, we see that $S$ is also of non-compact type and the total geodesy of $S$ follows from the theorem below.
\end{proof}

\begin{theorem} [{\rm See [14, p.14]}]
Let $X$ be a Hermitian locally symmetric manifold of non-compact type and $S$ be a Hermitian locally symmetric manifold of finite volume uniformized by an irreducible bounded symmetric domain of rank at least 2. Then, any non-constant holomorphic mapping from $S$ to $X$ is necessarily a totally geodesic isometric immersion up to a normalizing constant with respect to the canonical K\"ahler-Einstein metrics on $X$.
\end{theorem}

For every $X=\Omega/\Gamma$, with $\text{\rm rank}(\Omega)\geq 2$ and $\Omega$ irreducible, Theorem 2.2 gives an explicit upper bound of the dimension of any compact splitting complex submanifold of $X$ which fails to be totally geodesic. It turns out that, by combining the eigenbundle decomposition for $R_\tau$ on $T^1_1(X)$, we can obtain another upper bound which is much sharper in some cases.

\begin{lemma}
Let $M$ be a complex manifold and $x\in M$. Let $Q=(Q^{\alpha\beta}_{\mu\nu})\in {T^2_2}_{,x}(M)$ satisfying $Q^{\alpha\beta}_{\mu\nu}=Q^{\beta\alpha}_{\mu\nu}=Q^{\alpha\beta}_{\nu\mu}$. Regard $Q$ both as an element $Q_\odot\in\text{\rm End}(S^2T_x(M))$ and an element $Q_\star\in \text{\rm End}({T^1_1}_{,x}(M))$. If $Q_\odot=\lambda I_\odot$ for some non-zero $\lambda\in \mathbb C$, where $I_\odot$ is the identity map, then $\text{\rm Ker}(Q_\star)=\{0\}$.
\end{lemma}

\vspace{1.2mm}
\noindent
\begin{proof}
If $Q_\odot=\lambda I_\odot$, then
$$
    Q^{\alpha\beta}_{\mu\nu}=\frac{\lambda}{2}\left( \delta^\alpha_\mu\delta^\beta_\nu+\delta^\beta_\mu\delta^\alpha_\nu\right),
$$
where $\delta^\alpha_\mu$ is the Kronecker delta. Now let $T=(T^\mu_\alpha)\in {T^1_1}_{,x}(M)$, then
$$
    Q_\star (T) = \left(\sum_{\alpha,\mu} Q^{\alpha\beta}_{\mu\nu}T^\mu_\alpha\right)
    = \left(\frac{\lambda}{2}\sum_{\alpha,\mu} \left( \delta^\alpha_\mu\delta^\beta_\nu+\delta^\beta_\mu\delta^\alpha_\nu\right)T^\mu_\alpha\right)
    = \frac{\lambda}{2} \left(I\sum_\mu T^\mu_\mu  + T\right),
$$
where $I=(\delta^\beta_\nu)\in{T^1_1}_{,x}(M)$.

Now suppose $Q_\star(T)=0$. Since $\lambda\neq 0$, it follows that $T=-I\sum_\mu T^\mu_\mu$ and after taking trace, we get
$$
    \sum_\mu T^\mu_\mu = -\text{\rm dim}(M)\sum_\mu T^\mu_\mu.
$$
Thus, $\sum_\mu T^\mu_\mu = 0$ and hence $T=0$.
\end{proof}

\vspace{1.2mm}
\noindent
{\bf Remark.} \ An equivalent way of formulating the proof goes as follows.  $Q_{\mu\nu}^{\alpha\beta}$ is a curvature-like tensor.  The hypothesis $Q_\odot = \lambda I_\odot$ says that $Q$ agrees up to a non-zero multiplicative constant with the curvature tensor of the complex hyperbolic space form $(\mathbb B^n,g_{\mathbb B^n})$, and the conclusion Ker$(Q_\star) = 0$
follows from the fact that $(\mathbb B^n,g_{\mathbb B^n})$ is of strictly negative curvature in the dual sense of Nakano.

\vspace{1.2mm}
\noindent
\begin{proof}[Proof of Theorem 2.3]
We first define the endomorphism $R^S_\sigma:S^2T(S)\rightarrow S^2T(S)$, by composing the following mappings
$$
    S^2T(S)\overset{\sigma}{\longrightarrow} S^2T(X)\overset{R_\sigma}{\longrightarrow}S^2T(X)\overset{\pi_\sigma}{\longrightarrow} S^2T(S).
$$
Thus, $R^S_\sigma:=\pi_\sigma\circ R_\sigma\circ \sigma$.

Since $S^2T(X)=A\oplus B$ is the eigenbundle decomposition for $R_\sigma$, we have
$$R_\sigma=\lambda_Ap_A+\lambda_Bp_B,$$
where $p_A:S^2T(X)\rightarrow A$, $p_B:S^2T(X)\rightarrow B$ are the projections and $\lambda_A\neq 0$, $\lambda_B\neq 0$ are the two eigenvalues of $R_\sigma$. Hence, we get
$$
    R^S_\sigma=\lambda_AE_A +\lambda_BE_B.
$$

We now proceed with the proof. We may assume that $\text{\rm rank}(A)\leq \text{\rm rank}(B)$. Similar to the proof of Theorem 2.2, the hypothesis $\text{\rm rank}(S^2T(S))>\text{\rm rank}(A)$ implies that the endomorphism $E_A:S^2T(S)\rightarrow S^2T(S)$ must have a non-trivial kernel at every point on $S$. If $E_A$ is not identically equal to zero, then the same argument as in the proof of Theorem 2.2 shows that $S$ is Hermitian locally symmetric of rank at least 2 and totally geodesic.

\vskip 0.2cm
Now suppose $E_A\equiv 0$. Since $E_A+E_B=\text{\rm id}_{S^2T(S)}$, we have $E_B=\text{\rm id}_{S^2T(S)}$ and thus $$R^S_\sigma=\lambda_AE_A +\lambda_BE_B=\lambda_B\text{\rm id}_{S^2T(S)}.$$  By Lemma 4.2, if we regard $R^S_\sigma$ as an endomorphism $R^S_\star\in\text{\rm End}(T^1_1(S))$, then its kernel is  trivial at every point on $S$.

On the other hand, the curvature $(2,2)$-tensor on $X$ defines another endomorphism $R_\tau:T^1_1(X)\rightarrow T^1_1(X)$. Recall that by Borel~(Theorem 4.2), there is a decomposition $T^1_1(X)=C\oplus D$, where $C$ is the kernel of $R_\tau$ and $\text{\rm rank}(D)=\text{\rm dim}_{\mathbb R}\frak k$. Hence, if we define $R^S_\tau:T^1_1(S)\rightarrow T^1_1(S)$ by composing
$$
    T^1_1(S)\overset{\tau}{\longrightarrow} T^1_1(X)|_S \overset{R_\tau}{\longrightarrow} T^1_1(X)|_S \overset{\pi_\tau}{\longrightarrow} T^1_1(S),
$$
i.e. $R^S_\tau:=\pi_\tau\circ R_\tau\circ\tau$, then the hypothesis $\text{\rm dim}(S)^2 > \text{\rm dim}_{\mathbb R}\frak k$, which is equivalent to $\text{\rm rank}(T^1_1(S))> \text{\rm rank}(D)$, implies that $\tau(T^1_1(S))$ intersects $C|_S$ at every point on $S$ and thus $R^S_\tau$ has a non-trivial kernel. Finally, it is clear that $R^S_\tau$ and $R^S_\sigma$ are given by the same $(2,2)$-tensor on $S$ and it follows that $R^S_\tau = R^S_\star$. We have thus arrived at a contradiction and the proof of the theorem is complete.
\end{proof}

\section{Splitting complex submanifolds of quotients of irreducible bounded symmetric domains of type I and type IV}

In~\cite{s5} the multiplicities of the eigenvalues of $R_\sigma$ on $S^2T(X)$, where $X=\Omega/\Gamma$ for all classical types of irreducible bounded symmetric domains $\Omega$, are explicitly calculated.  We now use their results to specialize Theorem 2.3 to type-I and type-IV domains. We first recall the definitions of these domains.

\vspace{1.2mm}
\noindent
{\bf Type-I irreducible bounded symmetric domains $D^I_{m,n}$.}
{\it Let $M(m,n;\mathbb C)$ be the set of $m$-by-$n$ complex matrices. If we identify $M(m,n;\mathbb C)$ with $\mathbb C^{m\times n}$, then
$$
    D^I_{m,n}=\{Z\in M(m,n;\mathbb C): I_m-Z\bar Z^t>0\}\Subset\mathbb C^{m\times n}.
$$
}

\vspace{1.2mm}
\noindent
{\bf Type-IV irreducible bounded symmetric domains $D^{IV}_n$.}
{\it Let $n\geq 3$ be a positive integer. Then,
$$
    D^{IV}_n=\left\{(z_1,\ldots,z_n)\in\mathbb C^n: \sum^n_{k=1} |z_k|^2 < 2 \text{\rm\,\,\,\,\, and \,\,\,\,\,} \sum^n_{k=1} |z_k|^2 < 1 + \left|\frac{1}{2}\sum^n_{k=1} z_k^2\right|^2\right\}\Subset\mathbb C^n.
$$
}

\begin{theorem}
Let $p \ge 2$ be an integer, and $\, \Gamma \subset \text{\rm Aut}(D^I_{2,p})$ be a torsion-free discrete subgroup. Write $X = D^{I}_{2,p}/\Gamma$, and let $S\subset X$ be a compact splitting complex submanifold. If $\text{\rm dim}(S)\geq p+1$, then  $S$ is Hermitian locally symmetric of rank at least 2 and $S\subset X$ is totally geodesic with respect to the canonical K\"ahler-Einstein metrics on $X$. Furthermore, if $p\geq 3$, then $S$ is uniformized by $D^I_{2,q}$, for some $q \leq p$.
\end{theorem}

\vspace{1.2mm}
\noindent
\begin{proof}
Write $S^2T(X)=A\oplus B$, where $\text{\rm rank}(A)\leq \text{\rm rank}(B)$. By~\cite{s5},
$$
    \text{\rm rank}(A) = \frac{p(p-1)}{2} \quad \text{\rm and} \quad \text{\rm rank}(B) = \frac{3p(p+1)}{2}.
$$
On the other hand, it is well-known that the Lie algebra $\frak k$ of the isotropy group $K\subset\text{\rm Aut}(D^I_{2,p})$ at $0\in D^I_{2,p}$ can be identified with the Lie algebra of $S(U(2)\times U(p))$. Thus,
$$
    \text{\rm dim}_\mathbb R\frak k = \text{\rm dim}_\mathbb R \frak u(2) + \text{\rm dim}_\mathbb R \frak u(p) - 1
    = p^2+3.
$$
Hence, if $\text{\rm dim}(S)\geq p+1$, we have both
$$
\text{\rm rank}(S^2T(S))\geq \dfrac{(p+2)(p+1)}{2} > \text{\rm rank}(A)
$$
and $$\text{\rm dim}(S)^2 > \text{\rm dim}_\mathbb R\frak k.$$
Since the degree of the strong non-degeneracy of the bisectional curvature of $X=D^I_{2,p}/\Gamma$ is $p$ (\!\cite{s20}), it then follows from Theorem 2.3 that $S$ is Hermitian locally symmetric of rank at least 2 and totally geodesic with respect to the canonical K\"ahler-Einstein metrics on $X$. Finally, the last statement of the theorem follows from the classification of the totally geodesic symmetric submanifolds of $X$ (see \cite{s19}).
\end{proof}

\begin{theorem}
Let $n \ge 3$ be an integer, $\, \Gamma \subset \text{\rm Aut}(D^{IV}_n)$ be a torsion-free discrete subgroup. Write $X = D^{IV}_n/\Gamma$, and let $S\subset X$ be a compact splitting complex submanifold. If $\text{\rm dim}(S)>\dfrac{n}{\sqrt{2}}$, then  $S$ is Hermitian locally symmetric and uniformized by $D^{IV}_m$, $m\leq n$. Moreover, $S\subset X$ is totally geodesic with respect to the canonical K\"ahler-Einstein metric on $X$.
\end{theorem}

\vspace{1.2mm}
\noindent
\begin{proof}
By \cite{s5}, we can write $S^2T(X)=A\oplus B$, in which
$$
\text{\rm rank}(A)=1 \quad \text{\rm and} \quad \text{\rm rank}(B)=\dfrac{(n+2)(n-1)}{2}.
$$
Since the Lie algebra $\frak k$ of the isotropy group $K\subset\text{\rm Aut}(D^{IV}_n)$ at $0\in D^{IV}_n$ can be identified with the Lie algebra of $SO(n,\mathbb R)\times SO(2,\mathbb R)$, we have
$$
  \text{\rm dim}_\mathbb R\frak k = \text{\rm dim}_\mathbb R \frak {so}(n,\mathbb R) + \text{\rm dim}_\mathbb R \frak {so}(2, \mathbb R)
    = \frac{n(n-1)}{2}+1.
$$
Thus, if $\text{\rm dim}(S)>\dfrac{n}{\sqrt{2}}$, then we have both
$$
\text{\rm rank}(S^2T(S)) > 1 = \text{\rm rank}(A)
$$
and
$$\text{\rm dim}(S)^2 > \text{\rm dim}_\mathbb R\frak k.$$
Finally, the degree of the strong non-degeneracy of the bisectional curvature of $X=D^{IV}_n/\Gamma$ is $2$ (\cite{s20}) and the results now follow from Theorem 2.3 and the classification of the totally geodesic symmetric submanifolds of $X$ (see \,\cite{s19}).
\end{proof}

For $D^I_{2,p}$, $p\geq 2$, there is a totally geodesic symmetric subspace $\mathbb B^1\times \mathbb B^{p-1}\hookrightarrow D^I_{2,p}$ which can descend to their quotients. If there exists a surjective holomorphic map from a compact quotient of $\mathbb B^{p-1}$ to a compact quotient of $\mathbb B^1$ which is not an unramified covering map, then the graph of such mapping will give a $(p-1)$-dimensional compact splitting complex submanifold of $D^I_{2,p}/\Gamma$ for some $\Gamma$. It is also known that such maps exist for $p-1 = 1, 2, 3$ (see \cite{s6})
In view of this, we formulate our conjecture on $D^I_{2,p}$ as follows.

\begin{conj}
Let $p \ge 2$ be an integer and $\, \Gamma \subset \text{\rm Aut}(D^I_{2,p})$ be a torsion-free discrete subgroup. Write $X = D^{I}_{2,p}/\Gamma$, and denote by $g$ a canonical K\"ahler-Einstein metric on $X$.  Let $S \subset X$ be a compact splitting complex submanifold with $\dim(S)\ge p$.  Then, $(S,g|_S) \hookrightarrow (X,g)$ is totally geodesic.
\end{conj}

As we have just seen, when $X$ is uniformized by a type-IV domain $D^{IV}_n$, $n\geq 3$, we can write $S^2T(X)=A\oplus B$, where $\text{\rm rank}(A)=1$. So for a compact splitting complex submanifold $S\subset X$ with $\text{\rm dim}(S)\geq 2$, similar arguments show that either $S$ is Hermitian locally symmetric of rank at least 2 (which must be totally geodesic in $X$) or $E_A\equiv 0$. Here we recall that $E_A:S^2T(S)\rightarrow S^2T(S)$ is defined by composing
$$
    S^2T(S)\overset{\sigma}{\longrightarrow} S^2T(X)\overset{p_A}{\longrightarrow} A\overset{\pi_\sigma|_A}{\longrightarrow}S^2T(S).
$$
Clearly, if $\sigma(S^2T(S))\subset B$, then we have $p_A\circ\sigma\equiv 0$ and hence $E_A\equiv 0$. In such case, it means that $S$ is a characteristic complex submanifold in $X$ in the sense of Definition 2.3 (see \,Lemma 6.2 in the next section for a proof applicable to irreducible bounded symmetric domains of rank $\ge 2$ in general) and it follows from Theorem 2.1 that $S$ is totally geodesic. (This is the case of ``linear'' totally geodesic submanifolds.) At the same time, the analogous question for the dual Hermitian symmetric space, i.e., the hyperquadric, has been solved by \cite{s10}, where it was proven that any compact splitting complex submanifold $S \subset Q^n$ of dimension $\ge 2$ must either be a projective linear subspace or a smooth linear section which is itself a hyperquadric. In the latter case, endowing $Q^n$ with a K\"ahler-Einstein metric $h$, $(S,h|_S) \hookrightarrow (Q^n,h)$ is not necessarily totally geodesic. However, it is totally geodesic with respect to a K\"ahler-Einstein metric $h' = \gamma^*h$ for some $\gamma \in \text{\rm Aut}(Q^n)$. In the case of the noncompact dual, i.e., the type-IV domain $D^{IV}_n$ that we are considering, we suspect that a dual and more rigid situation is valid, since K\"ahler-Einstein metrics are unique up to scalar constants. We thus formulate the following conjecture regarding compact splitting complex submanifolds of dimension at least 2 for quotients of type-IV domains.

\begin{conj}
Let $n \ge 3$ be an integer, $\, \Gamma \subset \text{\rm Aut}(D^{IV}_n)$ be a torsion-free discrete subgroup, write $X = D^{IV}_n/\Gamma$, and denote by $g$ a canonical K\"ahler-Einstein metric on $X$.  Let $S \subset X$ be a compact splitting complex submanifold of dimension $\ge 2$.  Then, $(S,g|_S) \hookrightarrow (X,g)$ is totally geodesic.
\end{conj}

\section{A borderline case: compact splitting complex surfaces of quotients of the 3-dimensional and 4-dimensional Lie balls}

In what follows we examine compact splitting complex surfaces of quotients of type-IV domains (Lie balls) of dimension 3 and 4.  Since the problem reduces to that of $D^{IV}_4$ (by embedding $D^{IV}_3$ in a standard way in $D^{IV}_4$), and since $D^{IV}_4 \equiv D^{I}_{2,2}$, the case being studied is a borderline case both for Conjecture 5.1 and for Conjecture 5.2, and hopefully the partial results we have can shed some light on both conjectures.  We have

\begin{theorem}
Let $\, \Gamma \subset \text{\rm Aut}(D^{IV}_n)$ be a torsion-free discrete subgroup, $n = 3$ or $4$. Write $X = D^{IV}_n/\Gamma$, and let $g$ be a canonical K\"ahler-Einstein metric on $X$. Let $S \subset X$ be a 2-dimensional compact splitting complex submanifold, and let $T(X)|_S = T(S) \oplus {\cal N}$.  We have
\end{theorem}

\begin{enumerate}
\item[(1)] If $n = 3$, then $(S,g|_S) \hookrightarrow (X,g)$ is totally geodesic, and $S$ is biholomorphic to a quotient $\Delta^2/\Xi$ of the bidisk $\Delta^2$ by a cocompact torsion-free discrete lattice $\Xi \subset \text{\rm Aut}(\Delta^2)$.

\item[(2)]
 If $n = 4$, then either $(S,g|_S) \ \hookrightarrow (X,g)$ is totally geodesic, or ${\cal N}_x \subset T_x(X)$ is a characteristic $2$-plane for every point $x \in S$. In the totally geodesic case, $S$ is biholomorphic to $\Delta^2/\Xi$ as in $\text{\rm (1)}$ or to $\mathbb B^2/\Sigma$ for a cocompact torsion-free discrete lattice $\Sigma \subset \text{\rm Aut}(\mathbb B^2)$ on the complex $2$-ball $\, \mathbb B^2$.
\end{enumerate}

For the proof of Theorem 6.1 we need some basic facts about irreducible Hermitian symmetric manifolds.

\begin{lemma}
Let $(Z,g_c)$ be an irreducible Hermitian symmetric space of the compact type, where the underlying K\"ahler metric $g_c$ is such that minimal rational curves are {\rm(}totally geodesic and{\rm)} of constant Gaussian curvature $+2$, and denote by $\pi: {\mathscr C} (Z) \to Z$, ${\mathscr C} (Z) \subset \mathbb PT(Z)$ its canonical VMRT structure.  Then, for any $x \in Z$ and any unit vector $\alpha \in T_x(Z)$, $[\alpha] \in {\mathscr C}_x(Z)$ if and only if $W_{\alpha\ol{\alpha}\alpha\ol{\alpha}} = 2$, where $W$ is the curvature tensor of $(Z, g_c)$.
\end{lemma}

\vspace{1.2mm}
\noindent
\begin{proof}
By the Polysphere Theorem (see \cite{s24}) there exists a totally geodesic polysphere $P \subset Z$ of dimension $r = \text{\rm rank}(Z)$ passing through $x$ such that $\alpha \in T_x(P)$. $(P,g_c|_P)$ is isometrically biholomorphic to the Cartesian product of $r$ copies of $(\mathbb P^1,h)$, where $h$ is a Hermitian metric of constant Gaussian curvature $+2$ on the Riemann sphere $\mathbb P^1$. Identifying $\alpha$ as an element of $T_0(P)$ and writing thus
$\alpha = a_1e_1 + \cdots + a_re_r$ such that each $e_i$, $1 \le i \le r$, corresponds to a unit vector at $0 \in \mathbb P^1$
of the $i$-th Cartesian factor $(\mathbb P^1,h)$.  Then $|a_1|^2 + \cdots + |a_r|^2 = 1$ and
$W_{\alpha\ol{\alpha}\alpha\ol{\alpha}} = 2(|a_1|^4 + \cdots + |a_r|^4) = 2(|a_1|^2 + \cdots + |a_r|^2)^2 -
4\sum_{i < j}|a_ia_j|^2 \le 2$ and equality is attained if and only if $a_ia_j = 0$ whenever $i \neq j$, i.e., if and only if exactly one of the coefficients $a_i$ is non-zero (and of unit modulus), which is the case if and only if $[\alpha] \in {\mathscr C}_x(Z)$. 
\end{proof}

\begin{lemma}
Let $\Omega$ be an irreducible bounded symmetric domain of rank $\ge 2$.  For a reference point $0 \in \Omega$ denote by $K \subset G = \text{\rm Aut}_0(\Omega)$ the isotropy subgroup at $0$.  Denote by $S^2T_0(\Omega) = A_0 \oplus B_0$ the decomposition of the $K$-representation space $S^2T_0(\Omega)$ into two irreducible components, where $B_0$ is generated by the set of squares $\alpha\odot\alpha$ of highest weight vectors of $T_0(\Omega)$ as a $K$-representation space.  Suppose $V \subset T_0(\Omega)$ is a vector subspace such that $S^2V \subset B_0$.  Then $V \subset T_0(\Omega)$ is a characteristic vector subspace.
\end{lemma}

\vspace{1.2mm}
\noindent
\begin{proof}
Let $R$ be the curvature tensor of $(\Omega,g_\Omega)$. By Lemma 6.1 and by the duality between $(Z,g_c)$ and $(\Omega,g_\Omega)$ a unit vector $\chi \in T_0(\Omega)$ is a characteristic vector if and only if $R_{\chi\ol{\chi}\chi\ol{\chi}} = -2$.  On the other hand, denoting by $P$ the Hermitian bilinear form on $S^2T_0(\Omega)$ defined by $P(\xi\odot\mu,\eta\odot\nu) := R_{\xi\ol{\eta}\mu\ol{\nu}}$,
then $B_0 \subset S^2T_0(\Omega)$ is precisely the eigenspace of $P$ with eigenvalue $-2$.  Suppose now $S^2V \subset B_0$. Then, for any unit vector $\chi \in V$ we have $\chi\odot\chi \in B_0$ so that $R_{\chi\ol{\chi}\chi\ol{\chi}} = -2$.  But by Lemma 6.1 the latter holds if and only if $\chi$ is a characteristic vector, proving the lemma. 
\end{proof}

In what follows we will need to examine parallel transport of sets which are not necessarily vector spaces. For clarity we formalize the definition for parallel transport of sets, as follows.

\begin{defi}
For a Hermitian holomorphic vector bundle $(V,h)$ over $X$, $\pi: V \to X$,  with Hermitian connection $\nabla$, we say that a subset $Z \subset V$ is invariant under parallel transport to mean that given any point $x \in X$ any $\eta \in V_x \cap Z$, and any smooth curve $\gamma:(-a,a)\to X$ on $X$ passing through $x$, and for the smooth section $\widetilde \eta$ over $\gamma$ such that $\widetilde{\eta}(x) = \eta$ and such that $\nabla_{\dot\gamma}\widetilde\eta \equiv 0$, we must have $\widetilde\eta(y) \in Z$ for any point $y=\gamma(t)$, $-a<t<a$.
\end{defi}

The following lemma concerns parallel transport with respect to affine connections in general.  We will formulate it for K\"ahler manifolds for which the lemma will be applied in the current article.
We have

\begin{lemma}
Let $(X,g)$ be a K\"ahler manifold, $(V,h)$ be a Hermitian holomorphic vector bundle over $X$ and denote by $\nabla$ the Hermitian connection on $(V,h)$. Then, the following holds.
\end{lemma}

\begin{enumerate}
\item[(1)]
Suppose $E_1, E_2 \subset V$ are invariant under parallel transport, then $E_1 \cap E_2$ is invariant under parallel transport.
\item[(2)]
Denoting by ${\cal H} \subset S^2V$ the subset consisting of squares of non-zero tangent vectors, then ${\cal H}$ is invariant under parallel transport with respect to $\nabla$.
\item[(3)]
Suppose $E \subset V$ is a vector subbundle such that $S^2E \subset S^2V$ is invariant under parallel transport with respect to the Hermitian connection $\nabla$ on $(S^2V,S^2h)$.  Then $E \subset V$ is invariant under parallel transport on $(V,h)$.
\end{enumerate}

\vspace{1.2mm}
\noindent
\begin{proof}
(1) follows immediately from the definition of invariance under parallel transport.  For (2), given $x \in X$, a smooth curve $\gamma: (-a,a) \to X$, $\gamma(0) = x$, $\eta \in V_x$ and $\widetilde{\eta}$ a smooth section of $V$ over $\gamma$ satisfying $\widetilde\eta(x) = \eta$, we have
$\nabla_{\dot\gamma}(\widetilde\eta \otimes \widetilde\eta) = \nabla_{\dot\gamma}\widetilde\eta \otimes \widetilde\eta + \widetilde\eta \otimes \nabla_{\dot\gamma}\widetilde\eta = 2\nabla_{\dot\gamma}\widetilde\eta \odot \widetilde\eta$.
Hence, $\nabla_{\dot\gamma}\widetilde\eta = 0$ if and only $\nabla_{\dot\gamma}(\widetilde\eta \otimes \widetilde\eta) = 0$, i.e., ($\sharp$) $\widetilde\eta$ is invariant under parallel transport over $\gamma$ if and only if $\widetilde\eta \otimes \widetilde\eta$ is invariant under parallel transport over $\gamma$. The forward implication in $(\sharp)$ proves that ${\cal H}$ is invariant under parallel transport, giving (2).  For (3), if for each $\eta \in E_x$ we choose the $V$-valued smooth section $\widetilde\eta$ to be the parallel transport of $\eta$ over $\gamma$, for any $t \in (-a,a)$, the assignment $\eta \mapsto \widetilde\eta(\gamma(t))$ defines an injective linear map $\Phi_t: E_x \to V_{\gamma(t)}$.  Writing $H_t: = \Phi_t(E_x)$, the hypothesis under (3) implies by $(\sharp)$ that $S^2H_t = S^2E_{\gamma(t)}$, which implies that $H_t = E_{\gamma(t)}$.  ($\lambda \perp H_t \Rightarrow \lambda\odot\lambda\perp S^2H_t \Rightarrow \lambda\odot\lambda \perp S^2E_{\gamma(t)} \Rightarrow \lambda\perp E_{\gamma(t)}).$ Varying $x \in X$ and $\gamma$ this implies that $E$ is invariant under parallel transport, proving (3). 
\end{proof}

Regarding rigidity phenomena for compact complex submanifolds $S \subset X$ on quotient manifolds $X$ of type-IV domains (Lie balls), we have the following characterization result under the hypothesis that the holomorphic conformal structure on $X$ is non-degenerate when restricted to $S$, i.e., $S$ inherits a holomorphic conformal structure from $X$ by restriction.

\begin{theorem} [{\rm See \cite{s15}}]
Let $\, \Gamma \subset \text{\rm Aut}(D^{IV}_n), n \ge 3$, be a torsion-free discrete subgroup and write $X = D^{IV}_n/\Gamma$. Let $S \subset X$ be a compact complex submanifold of any dimension $d$, where $1\le d < n$, such that for any point $x \in S$, the restriction of the canonical holomorphic conformal structure on $D^{IV}_n$ is non-degenerate at $x$. Then, denoting by $g$ the canonical K\"ahler-Einstein metric on $X$, $(S,g|_S)$ is totally geodesic in $(X,g)$.
\end{theorem}

For the proof of Theorem 6.2 when $S$ is of dimension $\ge 3$ one made use of a result of \cite{s12}  on compact K\"ahler-Einstein manifolds admitting G-structures modeled on irreducible Hermitian symmetric spaces of the compact type, together with results on Hermitian metric rigidity of \cite{s13}, \cite{s14}. For the case in Theorem 6.2 where $S$ is of dimension 2 on top of Hermitian metric rigidity one made use of the following special result for holomorphic curves on quotients of type-IV domains.

\begin{lemma}[{\rm See [15, Lemma 1]}]
Let $n \ge 3$, $U \subset D^{IV}_n$ be an open subset, and $C \subset U$ be a connected smooth holomorphic curve such that $T_x(C)$ is spanned by a characteristic vector at any point $x \in C$. Suppose there exists on $C$ a parallel holomorphic
line subbundle $E \subset T(U)|_C$ spanned at each point $x \in C$ by a characteristic vector orthogonal to $T_x(C)$. Then, $C$ is a connected open subset of a minimal disk.
\end{lemma}

In the current article, for the proof of Theorem 6.1 we will also make use of the existence result on K\"ahler-Einstein metrics given by Proposition 2.1, the local characterization result for a certain type of totally geodesic holomorphic curves as given by Lemma 6.4, together with the following result related to Hermitian metric rigidity.

\begin{prop}[{\rm See\cite{s13}}]
Let $\Omega$ be a not necessarily irreducible bounded symmetric domain and $\Gamma \subset \text{\rm Aut}(\Omega)$ be a cocompact torsion-free discrete subgroup, $X = \Omega/\Gamma$.  Let $g$ be any $\text{\rm Aut}(\Omega)$-invariant K\"ahler metric $\text{\rm (}$which is of nonpositive holomorphic bisectional curvature$\text{\rm )}$, and denote by $R$ the curvature tensor of $(X,g)$.  Let $h$ be any Hermitian metric on $X$ of nonpositive curvature in the sense of Griffiths and denote by $\Theta$ the curvature tensor of $(X,h)$.  Then for any pair of $(1,0)$-tangent vectors $(\alpha,\zeta)$ at any point $x \in X$ which is a zero of holomorphic bisectional curvature of $(X,g)$, i.e., $R_{\alpha\ol{\alpha}\zeta\ol{\zeta}} = 0$, we must have $\Theta_{\alpha\ol{\alpha}\zeta\ol{\zeta}} = 0$.
\end{prop}

Proposition 6.1 is a consequence of Hermitian Metric Rigidity Theorem 
(\!\!\cite{s13}, \cite{s14}) in the non-compact case, including both the locally irreducible and the locally reducible case.  In the locally irreducible case, Proposition 6.1 for the case of characteristic vectors $\alpha$ was established first by an integral formula on first Chern forms and the uniqueness theorem for Hermitian metrics of nonpositive curvature in the sense of Griffiths was derived as a consequence.

For the bounded symmetric domain $D^{IV}_n$, $n \ge 3$, we denote by $D^{IV}_n \Subset \mathbb C^n \subset Q^n$ the standard embeddings incorporating the Harish-Chandra realization and the Borel embedding.  At $0 \in D^{IV}_n \subset Q^n$, the tangent space $T_0(D^{IV}_n) = T_0(Q^n)$ is endowed with a Hermitian inner product given by $g_\Omega$ and equivalently by $g_c$ as the dual canonical metrics agree at $0$. We have the following standard way for the expression of the canonical holomorphic conformal structure.

\begin{lemma}
Let $m \ge 2$ be an integer and $q \in \Gamma(Q^{2m},S^2T^*(Q^{2m})\otimes {\cal O}(2)))$ be an $\text{\rm Aut}(Q_{2m})$-invariant holomorphic section defining the canonical holomorphic conformal structure on $Q^{2m}$.  Let $e_1 \in T_0(D^{IV}_{2m})$ be a characteristic vector of unit length.  Then, there exists an orthonormal basis $\{e_1,\cdots,e_{2m}\}$ in $T_0(D^{IV}_{2m})$ consisting of characteristic vectors including $e_1$ such that, writing the corresponding dual basis as $\{e_1^*,\cdots,e_{2m}^*\}$ and identifying the fiber of ${\cal O}(2)$ at $0 \in Z$ with $\mathbb C$ by some linear isomorphism, we may write $q(0) = e_1^*\odot e_2^* + e_3^*\odot e_4^* + \cdots + e_{2m-1}^*\odot e_{2m}^*$, where $\odot$ denotes the symmetric tensor product.
\end{lemma}

We are now ready to give a proof of Theorem 6.1 characterizing compact complex splitting surfaces of quotient manifolds of Lie balls of dimension 3 or 4.

\vspace{1.2mm}
\noindent
\begin{proof}[Proof of Theorem 6.1.]
Consider first Case (1) of Theorem 6.1. Embedding $D^{IV}_3$ into $D^{IV}_4$ as a totally geodesic complex submanifold in the standard way accompanied by a Lie group monomorphism $\Psi: \text{\rm Aut}(D^{IV}_3) \to \text{\rm Aut}(D^{IV}_4)$, then $X= D^{IV}_3/\Gamma$ embeds as a totally geodesic complex submanifold of $X'= D^{IV}_4/\Gamma'$, $\Gamma' = \Psi(\Gamma)$. Case (1) for dimension $n = 3$ can then be deduced from the result for $n = 4$ as follows. Let $L$ be the orthogonal complement of $T(X)$ in $T(X')|_X$. Thus, $L$ is also a parallel line bundle on $X$. Let $S\subset X$ be a compact splitting complex surface and $\mathcal N\subset T(X)|_S$ be a holomorphic complementary bundle of $T(S)$ in $T(X)|_S$. Then $\mathcal N\oplus L|_S$ is a rank-2 holomorphic vector bundle on $S$ complementary to $T(S)$ in $T(X')|_S$. But $L$ is not generated by characteristic vectors on $X'$ and therefore the result in Case (2) for $n=4$ implies that $S$ is totally geodesic in $X'$ and hence totally geodesic in $X$. We remark that the only totally geodesic 2-dimensional complex submanifolds of the 3-dimensional Lie ball $D^{IV}_3$ are the maximal bidisks (this is easily seen if we identify $D^{IV}_4$ with $D^{I}_{2,2}$ and identify $D^{IV}_3$ as $D^{III}_{2} \subset D^{I}_{2,2}$ consisting of symmetric matrices) and the result for Case (1) now follows.

From now on we consider only the case $n = 4$.
Write $S^2T(X) = A \oplus B$ for the locally homogeneous holomorphic direct sum decomposition as given in Theorem 4.1, where $A$ is a holomorphic line bundle and $B$ is of rank 9, and denote by $\nu: S^2T(X) \to A$ the natural projection induced by the direct sum decomposition $S^2T(X) = A \oplus B$.
Denote by $\rho: D^{IV}_4 \to X = D^{IV}_4/\Gamma$ the universal covering map.  Let $x \in X$ be an arbitrary point, and let $\widetilde{x} \in D^{IV}_4$ be chosen such that $\rho(\widetilde x) = x$. Identify $T_x(X)$ with $T_0(D^{IV}_4)$ via a lifting of $T_x(X)$ to $T_{\tilde x}(D^{IV}_4)$ and an automorphism of $D^{IV}_4$. The locally homogeneous holomorphic line bundle $A \subset S^2T(X)$ is generated at each point $x \in X$ by an element $\frak a_x \in A_x$ identified with $e_1 \odot e_2 + e_3 \odot e_4 \in S^2T_0(D^{IV}_4)$.

Suppose $S^2T(S) \subset B|_S$. Then, by Lemma 6.2 we have $\mathbb PT(S) \subset {\mathscr C} (X)|_S$, and for each point $x \in S$, $T_x(S) \subset T_x(X)$ is a characteristic 2-plane.  By Theorem 2.1 it follows that $(S,g|_S) \hookrightarrow (X,g)$ is totally geodesic.  For the proof of Case (2) of Theorem 6.1, from now on we assume that $S^2T(S) \not\subset B|_S$. Denote by $\pi: T(X)|_S \to T(S)$ the holomorphic linear projection such that Ker$(\pi) = {\cal N}$ and $\pi|_{T(S)} = \text{\rm id}_{T(S)}$, and by $\pi_\sigma: S^2T(X)|_S \to S^2T(S)$ the holomorphic linear projection naturally induced by $\pi$.
Writing $\Phi := \pi_\sigma\circ\nu|_{S^2T(S)}$ we have a holomorphic bundle homomorphism given by
$$
\Phi: S^2T(S) \subset S^2T(X)|_S \overset{\nu}{\longrightarrow} A|_S \subset S^2T(X)|_S \overset{\pi_\sigma}{\longrightarrow} S^2T(S).
$$
Regard now $\Phi$ as a holomorphic bundle endomorphism of the rank-3 holomorphic vector bundle $S^2T(S)$. By Proposition 2.1, $S$ is equipped with a K\"ahler-Einstein metric (of negative Ricci curvature).  Hence, $T(S)$ is a holomorphic vector bundle over $S$ equipped with a Hermitian-Einstein metric, which also induces a Hermitian-Einstein metric on $S^2T(S)$. It follows that either (a) $S^2T(S)$ is a holomorphic and isometric direct sum of at least two Hermitian-Einstein holomorphic vector subbundles of rank $\ge 1$ of the same slope, or (b) the holomorphic vector bundle $S^2T(S)$ over $S$ is stable (see \cite{s21}). In Case (a) we claim that

\begin{enumerate}
\item[$(\dagger)$]
$S \subset X$ must be uniformized by the bi-disk $\Delta^2$, i.e., $S \cong \Delta^2/\Xi$ for some torsion-free discrete subgroup $\Xi \subset \text{\rm Aut}(\Delta^2)$, and moreover $(S,g|_S) \hookrightarrow (X,g)$ is totally geodesic.
\end{enumerate}

To establish $(\dagger)$ let $h$ be the unique K\"ahler-Einstein metric on $S$ of Ricci curvature $-2$. We will consider the holomorphic vector bundle $\tau: S^2T(S) \to S$ and its projectivization $\tau': \mathbb P(S^2T(S)) \to S$, and also holomorphic fiber subbundles (with not necessarily closed fibers) on them.  For any point $x \in S$, let ${\cal H}_x \subset S^2T_x(S)$ be the subset consisting of all $\xi \otimes \xi$ such that $\xi \in T_x(S), \xi \neq 0$. Varying over $x \in S$ we obtain a holomorphic fiber subbundle $\gamma : {\cal H} \to S$ (where $\gamma = \tau|_{{\cal H}}$) of $\tau: S^2T(S) \to S$ of fiber dimension 2.  By Lemma 6.3, ${\cal H} \subset S^2T(S)$ is invariant under parallel transport with respect to the Hermitian connection on $(S^2T(S),S^2h)$ induced from $(S,h)$.  Under the hypothesis for Case (a), the rank-3 Hermitian holomorphic vector bundle $(S^2T(S),S^2h)$ splits into a holomorphic and isometric direct sum of at least two Hermitian-Einstein holomorphic vector bundles of rank $\ge 1$ (and of the same slope).  In particular, there must be a rank-2 holomorphic vector subbundle $E \subsetneq S^2T(S)$ which is parallel with respect to $(S^2T(S),S^2h)$. For each point $x \in S$, $\mathbb P {\cal H}_x \subset \mathbb PS^2T_x(S)$ is a holomorphic curve of degree 2, where $\mathbb P{\cal H}_x$ denotes the image of ${\cal H}_x$ under the natural projection map $\beta_x: S^2T_x(S)-\{0\} \to \mathbb P(S^2T_x(S))$.  Varying $x$ over $S$ we will also write $\gamma': \mathbb P{\cal H} \to S$ for the corresponding holomorphic fiber subbundle of $\tau': \mathbb PS^2T(S) \to S$, where $\gamma' = \tau'|_{\mathbb P{\cal H}}$. For any point $x \in S$, $\mathbb P{\cal H}_x \cap \mathbb PE_x$ is nonempty, consisting either of two isolated points, or of a single unreduced point.  There exists therefore a nonempty connected open subset $U \subset S$ such that either (i) $\mathbb P{\cal H}_x \cap \mathbb PE_x$ consists of two isolated points for every $x\in U$ or (ii) $\mathbb P{\cal H}_x \cap \mathbb PE_x$ consists of a single unreduced point for every $x\in U$.  In either case, there exists a holomorphic section $\sigma: U \to \mathbb P{\cal H}\cap \mathbb PE$ which defines a holomorphic line subbundle $\Lambda \subset S^2T(S)|_U$, which is a parallel subbundle with respect to $(S^2T(S),S^2h)$.  Since every element of $\Lambda$ is a square $\xi\odot\xi$ it follows that there exists a
holomorphic line bundle $L \subset T(U)$ such that $\Lambda = S^2L \subset S^2T(S)$.  By Lemma 6.3, $L \subset T(U)$ is invariant under parallel transport.

Denoting by $L^{\perp} \subset T(U)$ the orthogonal complement of $L$ in $T(U)$, then $L^{\perp} \subset T(U)$ is also a parallel subbundle with respect to $(T(S),h)$, and $T(U) = L \oplus L^{\perp}$ is a holomorphic and isometric direct sum decomposition.  Hence, denoting by $R^h$ the curvature tensor of $(S,h)$, we have $R^h_{\xi\ol{\xi}\eta\ol{\eta}} = 0$ whenever $\xi \in L_x$ and $\eta \in L^{\perp}_x$ for some $x \in U$. Since $S$ is of complex dimension 2 and $(S,h)$ is the canonical K\"ahler-Einstein metric of Ricci curvature $-2$, $(U,h|_U)$ is locally isometric to the product of two Poincar\'e disks of Gaussian curvature $-2$.  In particular, $(U,h|_U)$ is locally symmetric, so that $\nabla R^h \equiv 0$, where $\nabla$ stands for the Hermitian connection of $(S,h)$.  Since $h$ is real-analytic, the condition $\nabla R^h \equiv 0$ holds everywhere on $S$, hence $(S,h)$ is Hermitian locally symmetric and of negative Ricci curvature, implying that $S$ is uniformized by the bidisk $\Delta^2$, i.e., $S \cong \Delta^2/\Xi$ for some torsion-free discrete subgroup $\Xi \subset \text{\rm Aut}(\Delta^2)$. Thus, $S$ is the image of a holomorphic embedding $f: \Delta^2/\Xi \ \hookrightarrow X =  D^{IV}_4/\Gamma$, and $f^*g$ is a K\"ahler metric of nonpositive holomorphic bisectional curvature on $\Delta^2/\Xi$.

For the proof of the claim $(\dagger)$ for Case (a) it remains to show that $(S,g|_S) \hookrightarrow (X,g)$ is totally geodesic.
Write $s$ for the K\"ahler metric $g|_S$.  It follows from Proposition 6.1 (related to Hermitian metric rigidity) that also $R^s_{\xi\ol{\xi}\eta\ol{\eta}} = 0$ whenever $\xi \in L_x$ and $\eta \in L^{\perp}_x$ for some $x \in U$.  Together with the Gauss equation this gives $0 = R^s_{\xi\ol{\xi}\eta\ol{\eta}} = R^g_{\xi\ol{\xi}\eta\ol{\eta}} - \|\sigma(\xi,\eta)\|^2 \le 0$, where $\sigma$ stands for the second fundamental form of $(S,g|_S) \hookrightarrow (X,g)$, so that also $R^g_{\xi\ol{\xi}\eta\ol{\eta}} = 0$, and $\sigma(\xi,\eta) = 0$. By Lemma 6.4, a local holomorphic integral curve $C$ of $L$ or of $L^{\perp}$ is necessarily totally geodesic, so that also $\sigma(\xi,\xi) = \sigma(\eta,\eta) = 0$, which together with $\sigma(\xi,\eta) = 0$ implies that $\sigma(x) = 0$ for every $x \in U$ and hence for every $x \in S$ by the Identity Theorem for real-analytic functions.  It follows that $\sigma \equiv 0$ on $S$, i.e., $(S,g|_S) \hookrightarrow (X,g)$ is totally geodesic, as claimed.

It remains to consider the situation where $S^2T(S) \not\subset B$, i.e., $\nu|_{S^2T(S)} \not\equiv 0$, and where Case (b) holds, i.e., $S^2T(S)$ is a stable holomorphic vector bundle, in which case $S^2T(S)$ must be simple as a holomorphic vector bundle over $S$, i.e., $\Gamma\big(S,\text{\rm End}(S^2T(S))\big) = \mathbb C\cdot \text{\rm id}_{S^2T(S)}$.  Since $\dim\big(\Phi(S^2T_x(S))\big) \le \text{\rm rank}(A) = 1 < 3 = \text{\rm rank}(S^2T(S))$ for any $x \in X$, it follows from the simplicity of $S^2T(S)$ that $\Phi \equiv 0$.  From $\nu|_{S^2T(S)} \not\equiv 0$ and rank$(A) = 1$ it follows that $\nu(S^2T_x(S)) = A_x$ for a general point $x \in S$.
From $\Phi := \pi_\sigma\circ\nu|_{S^2T(S)} = 0$ it follows that $\pi_\sigma(A_x) = 0$ for a general point $x \in S$,
and hence $\pi_\sigma(A_x) = 0$ for every point $x \in S$.  From now on $x \in S$ denotes an arbitrary point.
Since $\dim(\mathbb P{\cal N}_x) = 1$, $\dim({\mathscr C}_x(X)) = 2$ and $\mathbb PT_x(X)\cong \mathbb P^3$, $\mathbb P{\cal N}_x \subset \mathbb PT_x(X)$ and ${\mathscr C}_x(X)\subset \mathbb PT_x(X)$ must have nonempty intersection.  Under the aforementioned identification $T_x(X) \cong T_0(D^{IV}_4)$ without loss of generality we may assume $[e_1] \in \mathbb PN_x \cap {\mathscr C}_x(X)$. Then, we have
$$
0 = \pi_\sigma(e_1 \odot e_2 + e_3 \odot e_4) = \pi(e_1)\odot\pi(e_2) + \pi(e_3)\odot\pi(e_4) = \pi(e_3)\odot\pi(e_4) \, .
$$
It follows that either $\pi(e_3) = 0$ or $\pi(e_4) = 0$. Now $q(e_1,e_1) = q(e_1,e_3) = q(e_1,e_4) = q(e_3,e_3) = q(e_4,e_4) = 0$, and it follows that for $k = 3$ or $4$, and for $a, b \in \mathbb C$, we have $q(ae_1+be_k,ae_1+be_k) = a^2q(e_1,e_k) + 2abq(e_1,e_k) + b^2q(e_k,e_k) = 0$, implying that both $\mathbb Ce_1 + \mathbb Ce_3$ and $\mathbb Ce_1 + \mathbb Ce_4$ are characteristic 2-planes.
In other words, when $S^2T(S) \not\subset B$ and $S^2T(S)$ is a stable vector bundle over $S$, we have ${\cal N}_x \subset {\mathscr C}_x(X)$ for every point $x \in S$.  The proof of Theorem 6.1 is complete. 
\end{proof}

\vspace{1.2mm}
\noindent
{\bf Remarks.} 
\begin{enumerate}
\item[(a)]
For the last argument there is the following alternative verification. Identifying $D^{IV}_4$ with $D^{I}_{2,2}$, the quadratic form $q$ is given by $q\Big(\Big[
\smallmatrix
a & b\\
c & d
\endsmallmatrix\Big]\ ,
\Big[
\smallmatrix
a & b\\
c & d
\endsmallmatrix\Big]\Big)
= ad - bc$.
Identifying $e_1$ with $E_{11}$, $e_2$ with $E_{22}$, $e_3$ with $\sqrt{-1}E_{12}$ and $e_4$ with $\sqrt{-1}E_{21}$, where $E_{ij}$ is the matrix with the $(i,j)$-th entry being equal to 1 and all other entries being equal to 0, we have $q = e_1^* \odot e_2^* + e_3^*\odot e_4^*$.  Now, $\mathbb Ce_1 \oplus \mathbb Ce_3$ corresponds to $\mathbb CE_{11} \oplus \mathbb CE_{12}$, which consists of matrices of rank $\le 1$, so that $\mathbb Ce_1 \oplus \mathbb Ce_3 \subset T_0(D^{IV}_n)$ is a characteristic 2-plane.  The same applies to $\mathbb Ce_1 \oplus \mathbb Ce_4$, which corresponds to $\mathbb CE_{11} \oplus \mathbb CE_{21}$.

\item[(b)]
To prove $(\dagger)$, after showing that $S \subset X$ is the image of a holomorphic embedding of $\Delta^2/\Xi$, it follows readily that for $x \in S$, $T_x(S)$ agrees with $T_0(P)$ of a maximal bidisk $P \subset D^{IV}_4$, and hence the restriction of the canonical holomorphic structure of $X$ to $S$ is non-degenerate, from which $(\dagger)$ follows from Theorem 6.2.  We note however that beyond Hermitian metric rigidity the key additional argument in the proof of Theorem 6.2 for the case of surfaces $S$ was in fact Lemma 6.4, which is local in nature.
\end{enumerate}

\textbf{Acknowledgements.}
The first author was partially supported by the GRF grant 17303814 of the Research Grants Council of Hong Kong, China. The second author was partially supported by National Natural Science Foundation of China, grant No. 11501205 and Science and Technology Commission of Shanghai Municipality (STCSM), grant No. 13dz2260400.


\end{document}